\documentclass[a4 ,10pt]{article}
\usepackage[english]{babel}
\usepackage[utf8]{inputenc}
\usepackage{amssymb}
\usepackage{amsmath}
\usepackage[all]{xy}
\usepackage{xcolor}
\usepackage{auto-pst-pdf}
\usepackage{fancyhdr}
\def\R{\mathbb{ R}}
\def\N{\mathbb{ N}}
\newtheorem{theorem}{Theorem}[section]
\newtheorem{theorem and definition}{Theorem and definition}
\newtheorem{corollary}[theorem]{Corollary}

\newtheorem{lemma}[theorem]{Lemma}
\newtheorem{proposition}[theorem]{Proposition}
\newenvironment{proof}[1][Proof]{\textbf{#1.} }{\ \rule{0.5em}{0.5em}}

\title{On the Marchenko system and the long-time behavior of multi-soliton solutions of the one-dimensional Gross-Pitaevskii equation}
\date{\today}

\begin{document}

\author{
\renewcommand{\thefootnote}{\arabic{footnote}} Haidar Mohamad \footnotemark[1]}

\footnotetext[1]{Département de Mathématiques,
 Université Paris-Sud 11, F-91405 Orsay Cedex.  {\tt haidar.mohamad@math.u-psud.fr}}
\maketitle

\begin{abstract}

We establish a rigorous well-posedness results for the Marchenko system associated to the scattering 
theory of the one dimensional Gross-Pitaevskii equation (GP). Under some assumptions on the scattering data, these well-posedness results
provide regular solutions for (GP). We also construct particular solutions, called $N-$soliton solutions as an approximate superposition 
of traveling waves. A study for the asymptotic behaviors of such solutions when $t\rightarrow\pm \infty$ is also made.          
\end{abstract}

\section{Introduction}
In any space dimension $N$, it is usual to call nonlinear Schr\"odinger equation (NLS) the evolution equation
\begin{equation}\label{SN}
\begin{array}{lr}
i\partial_tu+\Delta u+f(|u|^2)u=0,& x\in \mathbb{R}^N,\quad t\in \R\end{array}
\end{equation}
where $f: \mathbb{R}^+\rightarrow \mathbb{R}$ is some smooth function. We often complete this equation by 
the boundary condition at infinity 
$$\lim_{|x|\rightarrow +\infty}|u(t,x)|^2 = \rho_0,$$ 
where the constant $\rho_0 \geq 0$ is such that $f(\rho_0)=0.$ Equation (NLS) is called focusing when $f'(\rho_0)>0$ and defocusing when $f'(\rho_0)<0.$

These equations are widely used as models in many domains of mathematical physics, especially in non-linear optic,  
superfluidity,  superconductivity and condensation of Bose-Einstein (see for example \cite{NS_superfluid, Bose_cond, Bose_Einstein_cond}). 
In non-linear optic case, the quantity $|u|^2$ corresponds to a light intensity, 
in that of the Bose-Einstein condensation, it corresponds to a density of condensed particles. 

In the following, we extend by Gross-Pitaevskii equation the one-dimensional non-linear Schr\"odinger equation 
with cubic nonlinearity, namely
\begin{equation}\label{GP_intro}\tag{GP}
i\partial_tu+\partial_x^2 u+(1-|u|^2)u=0,
\end{equation}
and for which the data at infinity satisfies $0\neq \rho_0 = 1$.

%%%%%%%%%%%%%%%%%%%%%%%%%%%%%%%%%%%%%%%%%%%%%%%%%%%%%%%%%%%%%%% 
\section*{Multi-soliton solutions}
%%%%%%%%%%%%%%%%%%%%%%%%%%%%%%%%%%%%%%%%%%%%%%%%%%%%%%%%%%%%%%% 

In the early seventies, Shabat and Zakharov discovered an integrable system for the one-dimensional cubic Schr\"odinger equation. 
They also proposed a semi-explicit method of resolution through the theory of inverse scattering. Their results were presented in two short articles, 
the first one is devoted to the focusing case \cite{Zakharof_Shabat_foc} and the second one is devoted to the defocusing case (i.e. to the
Gross-Pitaevskii equation \eqref{GP_intro}) \cite{Zakharof_Shabat}.
A similar structure was discovered a few years earlier by, Green, Kruskal and Miura \cite{GGKM} for Korteweg - de Vries equation
\begin{equation}\label{KDV}\tag{KdV}
 \partial_tv +\partial^3_xv- 6v\partial_xv = 0.
\end{equation}
These equations have in common special solutions called {\it solitons}, which play an important role in dynamics in the long time.  
In the following, we extend by soliton a progressive wave solution, i.e. a solution of the form
$$
u(t,x) = U(x-ct)
$$
where $U$ is the wave profile and $c$ its speed. For the Gross-Pitaevskii equation \eqref{GP_intro}, 
such solutions (of finite Ginzburg-Landau energy except for trivial constant solutions) exist if and only if the speed $c$ 
satisfies $c \in  (-\sqrt{2},\sqrt{2}).$ In this case, for a fixed speed, the profile is unique modulo the invariants of the equation: specifically
$$
u(t,x) = e^{i\theta_0} U_c(x-x_0-ct)
$$
where $U_c$ is explicitly given by 
\begin{equation}\label{soliton}
U_c(x) = \sqrt{1-\frac{c^2}{2}}\tanh\left(\sqrt{1-\frac{c^2}{2}} \frac{x}{\sqrt{2}}\right)+ i\frac{c}{\sqrt{2}}
\end{equation}
and  $\theta_0,$  $x_0 \in \R$ are arbitrary and reflect the invariance by renormalization of the phase and translation. 

We will prove the following result : 
 \begin{theorem}\label{GP_copm}
Let $N\geq 1$, $c_1 < \cdots < c_N \in (-\sqrt{2},\sqrt{2})$, $\theta_0\in \R$ and $x_1,\cdots,x_N\in \R$ be arbitrary and fixed.  There exists
a smooth solution $u$ for the Gross-Pitaevskii equation \eqref{GP_intro} on whole $\R\times \R$ such that
$$
\lim_{t\rightarrow \pm\infty}u(t,x + c_k t) = e^{i(\theta_0+\theta_k)} A_{k}^{\pm}U_{c_k}(x - x_k - \eta_{k}^{\pm}),$$
for all $x \in \mathbb{R}$ and $k\in \{1,\cdots,N\},$ with 
$$\eta_{k}^+ = -\frac{1}{\sqrt{2-c_k^2}}\sum_{j=k+1}^{N}\log\left(\frac{2-c_jc_k + \sqrt{2-c_j^2}\sqrt{2-c_k^2}}{2-c_jc_k-\sqrt{2-c_j^2}\sqrt{2-c_k^2}
}\right)\quad , \quad
A_{k}^+ =e^{2i\sum_{j=k+1}^{N}\theta_j},$$
$$\eta_{k}^- = -\frac{1}{\sqrt{2-c_k^2}}\sum_{j=1}^{k-1}\log\left(\frac{2-c_jc_k + \sqrt{2-c_j^2}\sqrt{2-c_k^2}}{2-c_jc_k-\sqrt{2-c_j^2}\sqrt{2-c_k^2}
}\right)\quad , \quad
A_{k}^- =e^{2i\sum_{j=1}^{k-1}\theta_j},$$
and
$$\theta_j = \arccos(\frac{c_j}{\sqrt{2}}),\quad j=1,\cdots,N.$$
\end{theorem}

In other words, in each coordinate system in translation with speed $c_k$, the solution $u$ behaves asymptotically when
$t \to \pm \infty$ like the soliton $U_{c_k}$ translated appropriately, and the translation factors when  
$t = \pm \infty$, given by $\eta_k^+$ and $\eta_k^-,$ can be interpreted as a shift of binary collisions between the solitons.    
A similar result for the focusing equation was described in \cite{Zakharof_Shabat_foc},  
but the equivalence of Theorem \ref{GP_copm} has not been proved by \cite{Zakharof_Shabat}. 

The solutions of Theorem \ref{GP_copm} are obtained via an integrable system 
\footnote{More exactly, through a family of integrable system depending on two real variables.} called Marchenko system. 

Let $\beta :\mathbb{R}\rightarrow \mathbb{C}$ and $N\in \mathbb{N}^*.$
For $\xi \in \mathbb{R}$, we set
$$\lambda =\lambda(\xi)=\sqrt{\xi^2+\frac{1}{2}},$$
$$c_1(\xi)=\beta(\lambda)+\beta(-\lambda),$$
$$c_2(\xi)=\frac{\beta(\lambda)-\beta(-\lambda)}{\lambda}.$$
Assume that the function $\xi \mapsto \beta(\lambda(\xi))$ belongs to the space $L^1(\mathbb{R})\cap L^2(\mathbb{R}).$ 
Consider: 
\begin{enumerate}
 \item $N$ real numbers $-\frac{1}{\sqrt{2}}< \lambda_1<...<\lambda_N < \frac{1}{\sqrt{2}}.$
 \item $N$ real strictly negative numbers  $\mu_1,..., \mu_N$.
\end{enumerate}
Finally,  we define the functions of real variable
$$F_{11}(z) = \sum_{k=1}^{N}\mu_k\lambda_ke^{-\nu_kz},\quad F_{12}(z) = \frac{1}{2\pi}\int_{-\infty}^{+\infty}c_1(\xi)e^{i\xi z}d\xi,$$
$$F_{21}(z) = \sum_{k=1}^{N}\mu_ke^{-\nu_kz},\quad F_{22}(z) = \frac{1}{2\pi}\int_{-\infty}^{+\infty}c_2(\xi)e^{i\xi z}d\xi,$$
$$F_1 = F_{12}-F_{11},\quad F_2 = F_{22}-F_{21},$$
with $\nu_k =\sqrt{\frac{1}{2}-\lambda_k^2}, k\in {1,...,N}.$

\begin{proposition}\label{wp_Marchenko_1} 
Let $n\geq 0$ be such that 
\begin{equation}\label{hyp_beta_intro}
 \lambda\mapsto \lambda^{n+2}\beta(\lambda) \in L^\infty([\frac{1}{\sqrt{2}},+\infty[).
\end{equation}
There exists a constant $\epsilon_0>0$, depending on $N$, $n$, $\lambda_i$ and $\mu_i$, such that if
$$\|\lambda^n\beta(.)\|_{L^\infty([\frac{1}{\sqrt{2}},+\infty[)} \leq \epsilon_0$$ 
then for every $x\in \mathbb{R}$ the Marchenko system (with parameter $x$)
\begin{equation}\label{marchenko1}
\left\{\begin{array}{lr}
2\sqrt{2}\Psi_1(x,y)=F_2(x+y)\\
\quad \quad \quad\quad\quad\quad 
-\int_{x}^{+\infty}[\sqrt{2}\Psi_2(x,s)(F_1(s+y)-iF'_2(s+y))+\Psi_1(x,s)F_2(s+y)]ds,\\
2\sqrt{2}\Psi_2(x,y)=\sqrt{2}(F_1(x+y)+iF'_2(x+y))\\
\quad \quad \quad\quad\quad \quad
-\int_{x}^{+\infty}[\sqrt{2}\Psi_1(x,s)(F_1(s+y)+iF'_2(s+y))+\Psi_2(x,s)F_2(s+y)]ds,
\end{array}\right.
\end{equation}
has a unique solution 
$\Psi(x,.)\in (H^n_y(y\geq x))^2$. Moreover,  the function
$$\underline{\Psi}:\mathbb{R}\times \mathbb{R}^+\rightarrow \mathbb{C}^2: \underline{\Psi}(x,p)= \Psi(x,x+p),$$
belongs to the space $\mathcal{C}^k(\mathbb{R},(H^{n-k}(\mathbb{R}^+))^2),\quad k\in \{0,...,n-1\}.$
\end{proposition}

The well-posedness in the previous proposition are obtained by perturbation of the case $\beta \equiv 0$, which
corresponds to the exact multi-solitons solutions, and for which the result is completely of algebraic nature.

We refer the reader to \cite{Ger-Zng} and \cite{th_scat} for more details about the derivation of 
Marchenko system via the scattering theory associated to (GP). 
We focus in this paper on the other direction where 
the link with the Gross-Pitaevskii equation is obtained by the following proposition which introduces an  additional time parameter
$t$ in Marchenko system:

\begin{proposition}\label{wp_Marchenko_2}
Assume that    
$$\beta(t,\lambda)=c(\lambda)\exp(-4i\lambda\xi t),\quad \mu_k(t)=\mu^0_k\exp(4\lambda_k\nu_kt),\quad 1\leq k\leq N,$$
where the $\mu^0_k$ are assumed to be strictly negative and the function $c$ satisfies assumptions  \eqref{hyp_beta_intro} above.  
Then if the quantity  $\|\lambda^nc(.)\|_{L^\infty([\frac{1}{\sqrt{2}},+\infty[)}$ is small enough, 
for every $(t,x)\in \mathbb{R}^2$ the Marchenko system (with fixed parameters $x$ and $t$) has a unique solution 
$\Psi(t,x,.)\in (H^n_y(y\geq x))^2$ such that
$\underline{\Psi} \in \mathcal{C}^k(\mathbb{R}^2,(H^{n-2k}(\mathbb{R}^+))^2)$ for all $k\in \mathbb{N}$ with $ k<\frac{n}{2}.$ 
Moreover, if $n\geq 3$, the function 
$$u(t,x)=1+2\sqrt{2}i\underline{\overline{\Psi}}_{12}(t,x,0)$$
is solution of the Gross-Pitaevskii equation in  $\mathcal{C}^1(\mathbb{R}^2)$ such that $u(t,.) \in \mathcal{C}^2(\mathbb{R})$ for all $t\in \mathbb{R}.$ 
Finally, if the $\nu_k$ are pairwise distinct and $c$ is a real-valued function,  for $t\in \mathbb{R}$ fixed, we have
$$\lim_{x\rightarrow \pm\infty} (u(t,x)-u_N(t,x)) = 0,$$
where $u_N$ is the solution corresponding to $\beta \equiv 0$ (the other coefficients $\lambda_k$ and $\mu_k^0$ are left equal).  
\end{proposition} 

The solutions mentioned in Theorem \ref{GP_copm} are exactly the above solutions $u_N$. The link between 
the speeds $c_k$ of Theorem \ref{GP_copm} and the coefficients $\lambda_k$ of Marchenko system is given by the relationships $c_k = 2\lambda_k.$

%%%%%%%%%%%%%%%%%%%%%%%%%%%%%%%%%%%%%%%%%%%%%%%%%%%%%PROPOSITION réciproque %%%%%%%%%%%%%%%%%%%%%%%%%%%%%%%%%%%
\section{$N$-soliton solutions}\label{sec_N_solit}
The purpose of this section is the resolution of Marchenko system when $\beta(\lambda)=0.$  We set 
\begin{enumerate}
 \item $N$ real numbers $-\frac{1}{\sqrt{2}} <\lambda_1<\lambda_2 <...<\lambda_N <\frac{1}{\sqrt{2}}$.
 \item $N$ negative real numbers $\mu_1,\mu_2,...,\mu_N$.
\end{enumerate}
In this case, the Marchenko system is rewritten as follows
\begin{eqnarray}
2\sqrt{2}\Upsilon_1(x,y)&=& -\sum_{k=1}^{N}\mu_ke^{-\nu_kx}e^{-\nu_ky}\notag\\
&& +\sum_{k=1}^{N}e^{-\nu_ky}\int_{x}^{+\infty}[\sqrt{2}\Upsilon_2(x,s)\mu_k(\lambda_k+i\nu_k)e^{-\nu_ks}+\Upsilon_1(x,s)\mu_ke^{-\nu_ks}]ds,\notag\\
\label{marchenko22}
\end{eqnarray}
\begin{eqnarray}
 2\sqrt{2}\Upsilon_2(x,y)&=& \sqrt{2}\sum_{k=1}^{N}\mu_k(-\lambda_k+i\nu_k)e^{-\nu_kx}e^{-\nu_ky}\notag\\
&& +\sum_{k=1}^{N}e^{-\nu_ky}\int_{x}^{+\infty}[\sqrt{2}\Upsilon_1(x,s)\mu_k(\lambda_k-i\nu_k)e^{-\nu_ks}+\Upsilon_2(x,s)\mu_ke^{-\nu_ks}]ds.\notag\\
\label{marchenko33}
\end{eqnarray}
This means that  $\Upsilon_1$ and $ \Upsilon_2$ take the forms
$$\Upsilon_1(x,y)=\sum_{j=1}^{N}f_j(x)e^{-\nu_jy},\quad\Upsilon_2(x,y)=\sum_{j=1}^{N}g_j(x)e^{-\nu_jy}.$$
Denote $\alpha_k=\mu_k(\lambda_k-i\nu_k)$, hence 
\begin{eqnarray*}
\int_{x}^{+\infty}\Upsilon_2(x,s)e^{-\nu_ks}ds &=& \int_{x}^{+\infty}\sum_{j=1}^{N}g_j(x)e^{-(\nu_j+\nu_k)s}ds  \\
&=& \sum_{j=1}^{N}g_j(x)\frac{e^{-(\nu_j+\nu_k)x}}{\nu_j+\nu_k}.
\end{eqnarray*}
Similarly, we have $$\int_{x}^{+\infty}\Upsilon_1(x,s)e^{-\nu_ks}ds=\sum_{j=1}^{N}f_j(x)\frac{e^{-(\nu_j+\nu_k)x}}{\nu_j+\nu_k},$$ and the two
equations (\ref{marchenko22}) and (\ref{marchenko33}) become
\begin{eqnarray*}
2\sqrt{2}\Upsilon_1(x,y)&=& -\sum_{k=1}^{N}\mu_ke^{-\nu_kx}e^{-\nu_ky} \\
&& +\sum_{k=1}^{N}e^{-\nu_ky}\sum_{j=1}^{N}[\sqrt{2}\bar{\alpha}_kg_j(x)+\mu_kf_j(x)]\frac{e^{-(\nu_j+\nu_k)x}}{\nu_j+\nu_k},
\end{eqnarray*}
\begin{eqnarray*}
2\sqrt{2}\Upsilon_2(x,y)&=& -\sqrt{2}\sum_{k=1}^{N}\alpha_ke^{-\nu_kx}e^{-\nu_ky} \\
&& +\sum_{k=1}^{N}e^{-\nu_ky}\sum_{j=1}^{N}[\sqrt{2}\alpha_kf_j(x)+\mu_kg_j(x)]\frac{e^{-(\nu_j+\nu_k)x}}{\nu_j+\nu_k}.
\end{eqnarray*}
Identifying the terms with factors $\exp(-\nu_ky)$, we get 
\begin{equation}
\label{marchenko222}
2\sqrt{2}f_k(x)-\sum_{j=1}^{N}[\mu_kf_j(x)+\sqrt{2}\bar{\alpha}_kg_j(x)]\frac{e^{-(\nu_j+\nu_k)x}}{\nu_j+\nu_k}=-\mu_ke^{-\nu_kx},
\quad k\in \{1,2,...,N\}.
\end{equation}

\begin{equation}
\label{marchenko333}
2\sqrt{2}g_k(x)-\sum_{j=1}^{N}[\sqrt{2}\alpha_kf_j(x)+\mu_kg_j(x)]\frac{e^{-(\nu_j+\nu_k)x}}{\nu_j+\nu_k}=-\sqrt{2}\alpha_ke^{-\nu_kx},
\quad k\in \{1,2,...,N\}
\end{equation}
Assuming $x$ as a parameter, the $2N$  equations (\ref{marchenko222}) and (\ref{marchenko333}) correspond to linear system with 
$2N$ equations and $2N$ unknowns. To analyze its resolution, we set 

$$  T=\frac{1}{2\sqrt{2}}
\left( \begin{matrix}
A&\overline{B}\\
B&A\\
\end{matrix} \right ),
$$
$$C=\frac{1}{2\sqrt{2}} \left( \begin{matrix}
C^1\\
C^2\\
 \end{matrix} \right )
,\quad F=
\left( \begin{matrix}
F^1\\
F^2\\
\end{matrix} \right ),
$$
with $$\forall i,j \in \{1,2,..,N\}\quad (A)_{ij}=-\mu_i\frac{e^{-(\mu_i+\mu_j)x}}{\mu_i+\mu_j}\quad \forall i,j \in \{1,2,..,N\},$$
$$ (B)_{ij}=-\sqrt{2}\alpha_i\frac{e^{-(\mu_i+\mu_j)x}}{\mu_i+\mu_j}\triangleq-\beta_i\frac{e^{-(\mu_i+\mu_j)x}}{\mu_i+\mu_j}
\quad \forall i,j \in \{1,2,..,N\},$$
$$(C^1)_k=-\mu_ke^{-\nu_k x},\quad (C^2)_k=-\sqrt{2}\alpha_ke^{-\nu_k x}\quad \forall k\in \{1,2,..,N\},$$
$$(F^1)_k=f_k,\quad (F^2)_k=g_k\quad \forall k\in \{1,2,..,N\},$$

so that (\ref{marchenko222}) and (\ref{marchenko333}) can be rewritten in the matrix form: 
\begin{equation}\label{sys1}
(I+T)F=C.
\end{equation}

%%%%%%%%%%%%%%%%%%%%%% existence et unicité de solution à N solitons %%%%%%%%%%%%%%%%%
\begin{proposition}
The matrix $(I+T)(x)$ is invertible for every $x\in\mathbb{R}$.
\end{proposition}
\begin{proof}
The idea of this proof is to demonstrate that $-1$ is not an eigenvalue for $T$. To this end, define 
$$  Q=
\left( \begin{matrix}
I_N&0\\
D&I_N\\
\end{matrix} \right ),
$$
with $D \in \mathcal{M}_N(\mathbb{C})$  is a diagonal matrix defined as follows $(D)_{ii}=-\frac{\beta_i}{\mu_i}\quad \forall i\in \{1,...,N\}.$
Using the fact that $|\beta_i|^2=\mu_i^2$, we get 
$$  QTQ^{-1}=\frac{1}{2\sqrt{2}}
\left( \begin{matrix}
A-D\overline{B}&\overline{B}\\
0&0\\
\end{matrix} \right ),
$$
hence we have the following equivalence 
 $$\neg(-1\in Sp(T))\Leftrightarrow \neg(-1\in Sp(A-D\overline{B}))\Leftrightarrow\det(A-D\overline{B}+I_N)(x)\neq 0,\quad \forall x\in \mathbb{R},$$
where 
$$  A-D\overline{B} =
 -\begin{pmatrix}
(\mu_1+\bar{\beta}_1\frac{\beta_1}{\mu_1})\frac{e^{-2\nu_1x}}{2\nu_1}&(\mu_1+\bar{\beta}_1\frac{\beta_2}{\mu_2})\frac{e^{-(\nu_1+\nu_2)x}}{\nu_1+\nu_2}&...&(\mu_1+\bar{\beta}_1\frac{\beta_N}{\mu_N})\frac{e^{-(\nu_1+\nu_N)x}}{\nu_1+\nu_N}\\
(\mu_2+\bar{\beta}_2\frac{\beta_1}{\mu_1})\frac{e^{-(\nu_1+\nu_2)x}}{\nu_1+\nu_2}&(\mu_2+\bar{\beta}_2\frac{\beta_2}{\mu_2})\frac{e^{-2\nu_2x}}{2\nu_2}&...&(\mu_2+\bar{\beta}_2\frac{\beta_N}{\mu_N})\frac{e^{-(\nu_2+\nu_N)x}}{\nu_2+\nu_N}\\
...&...& &...\\
(\mu_N+\bar{\beta}_N\frac{\beta_1}{\mu_1})\frac{e^{-(\nu_1+\nu_N)x}}{\nu_1+\nu_N}&(\mu_N+\bar{\beta}_N\frac{\beta_2}{\mu_2})\frac{e^{-(\nu_2+\nu_N)x}}{\nu_1+\nu_2}&...&(\mu_N+\bar{\beta}_N\frac{\beta_N}{\mu_N})\frac{e^{-2\nu_Nx}}{\nu_N}
\end{pmatrix}.
$$
Denote now $\hat{D}=-diag(\mu_1,\mu_2,...,\mu_N)$. Then  we have $$\det(A-D\overline{B}+I_N)\neq 0\Leftrightarrow\det((A-D\overline{B}+I_N)\hat{D})\neq 0,$$
since $\det(\hat{D})=\prod_{k=1}^{N}(-\mu_k) >0$, and 

$$  (A-D\overline{B})\hat{D} =
 \begin{pmatrix}
(\mu_1^2+\bar{\beta}_1\beta_1)\frac{e^{-2\nu_1x}}{2\nu_1}&(\mu_1\mu_2+\bar{\beta}_1\beta_2)\frac{e^{-(\nu_1+\nu_2)x}}{\nu_1+\nu_2}&...&(\mu_1\mu_N+\bar{\beta}_1\beta_N)\frac{e^{-(\nu_1+\nu_N)x}}{\nu_1+\nu_N}\\
(\mu_2\mu_1+\bar{\beta}_2\beta_1)\frac{e^{-(\nu_1+\nu_2)x}}{\nu_1+\nu_2}&(\mu_2^2+\bar{\beta}_2\beta_2)\frac{e^{-2\nu_2x}}{2\nu_2}&...&(\mu_2\mu_N+\bar{\beta}_2\beta_N)\frac{e^{-(\nu_2+\nu_N)x}}{\nu_2+\nu_N}\\
...&...& &...\\
(\mu_N\mu_1+\bar{\beta}_N\beta_1)\frac{e^{-(\nu_1+\nu_N)x}}{\nu_1+\nu_N}&(\mu_N\mu_2+\bar{\beta}_N\beta_2)\frac{e^{-(\nu_2+\nu_N)x}}{\nu_1+\nu_2}&...&(\mu_N^2+\bar{\beta}_N\beta_N)\frac{e^{-2\nu_Nx}}{\nu_N}
\end{pmatrix}.
$$
It is clear that the matrix $(A-D\overline{B})\hat{D}$ is Hermitian and that 
$$ \left((A-D\overline{B})\hat{D}\right)_{ij}=(\mu_i\mu_j+\bar{\beta}_i\beta_j)\int_{-\infty}^{-x}e^{(\nu_i+\nu_j)t}dt\quad \forall i, j \in \{1,2,..,N\}.$$
Let now $Z\in \mathbb{C}^N\backslash\{0\}$. Then 
\begin{eqnarray*}
<(A-D\overline{B}+I_N)\hat{D}Z,Z>&=&<(A-D\overline{B})\hat{D}Z,Z>+<\hat{D}Z,Z>\\
&& \int_{-\infty}^{-x}|\sum_{i=1}^{N}\mu_iz_ie^{\nu_it}|^2dt+\int_{-\infty}^{-x}|\sum_{i=1}^{N}\beta_iz_ie^{\nu_it}|^2dt-\sum_{i=1}^{N}\mu_i|z_i|^2 \\
&& > 0, \\
\end{eqnarray*}
which means that the matrix $(A-D\overline{B}+I_N)\hat{D}$ is positive definite, hence $\det(A-D\overline{B}+I_N)\hat{D}\neq 0$.  Thus, the proof is completed.
\end{proof}

%%%%%%%%%%%%%%%%%%%%%%%%%%%%%%%%%%%%%%%%%%%%%%%%%%%%%%%
\section{Resolution of Marchenko system: proofs of Propositions \ref{wp_Marchenko_1} and \ref{wp_Marchenko_2}}
 
\subsection{The case with no time dependence}
 
The solution corresponding to $\beta \neq 0$ will be obtained by perturbation of the $N$-soliton solution constructed 
in the previous section.
Denote $\Upsilon = \left(\begin{matrix}\Upsilon_1\\ \Upsilon_2\end{matrix}\right)$ the solution of Marchenko system
\eqref{marchenko1} corresponding to $\beta \equiv 0$ and the same values of the other scattering coefficients, 
i.e. $\lambda_i$ and $\mu_i.$ Define 
$$\Psi^r_1 = \Psi_1-\Upsilon_1, \quad \Psi^r_2 = \Psi_2-\Upsilon_2.$$
The Marchenko system \eqref{marchenko1} is thus transformed into a system for 
$\Psi^r= \left(\begin{array}{lr} \Psi^r_1\\ \Psi^r_2\end{array}\right):$ 
\begin{equation}\label{wp}
 (2\sqrt{2}I_d+\Omega_x)\Psi^r(x,y) =  \mathcal{F}\Psi^r(x,y),
\end{equation}
where the two components $(\Omega_x\Psi^r)_1(x,.)$ and $(\Omega_x\Psi^r)_2(x,.)$ of $\Omega_x\Psi^r(x,y)$ are given by
$$ (\Omega_x\Psi^r)_1(x,y) =-\int_x^{+\infty}\sqrt{2}\Psi^r_2(x,s)(F_{11}-iF'_{21})(s+y)ds -\int_x^{+\infty}\Psi^r_1(x,s)F_{21}(s+y)ds,$$
$$(\Omega_x\Psi^r)_2(x,y)= -\int_x^{+\infty}\sqrt{2}\Psi^r_1(x,s)(F_1+iF'_{21})(s+y)ds -\int_x^{+\infty}\Psi^r_2(x,s)F_{21}(s+y)ds,$$
and the two components $(\mathcal{F}\Psi^r)_1(x,.)$ and $(\mathcal{F}\Psi^r)_2(x,.)$ of $\mathcal{F}\Psi^r(x,.)$ are given by
\begin{eqnarray}\label{wp3}
 (\mathcal{F}\Psi^r)_1(x,y)&=& -\int_x^{+\infty}\sqrt{2}(\Upsilon_2+\Psi^r_2)(x,s)(F_{12}-iF'_{22})(s+y)ds\notag\\
 && + \int_x^{+\infty}(\Upsilon_1+\Psi^r_1)(x,s)F_{22}(s+y)ds + F_{22}(x+y),
\end{eqnarray}
\begin{eqnarray}\label{wp4}
(\mathcal{F}\Psi^r)_2(x,y)&=& -\int_x^{+\infty}\sqrt{2}(\Upsilon_1+\Psi^r_1)(x,s)(F_{12}+iF'_{22})(s+y)ds\notag\\
 && + \int_x^{+\infty}(\Upsilon_2+\Psi^r_2)(x,s)F_{22}(s+y)ds + \sqrt{2}(F_{12}+iF'_{22})(x+y).
\end{eqnarray}
The operator $\Omega_x\in \mathcal{L}(L^2_y(y\geq x))$ satisfies the following property: 
\begin{lemma}\label{op_Omega}
 The operator $\Omega_x$ is positive independently of $x$. 
 Namely, 
 $$\int_x^{+\infty}\Omega_x\Phi\cdot\bar{\Phi} \geq 0, \qquad \forall \Phi \in (L^2_y(y\geq x))^2.$$ 
\end{lemma}
\begin{proof}
 In view of the definition of $F_{11}$ and $F_{21},$ we have  
 \begin{eqnarray}\label{op_Omega1}
 \int_x^{+\infty}(\Omega_x\Phi)_1\bar{\Phi}_1 &=& -\sum_{k=1}^{N}\mu_k|\int_x^{+\infty}e^{-\nu_ky}\Phi_1(y)dy|^2 \notag\\ 
 &&- \sqrt{2}\sum_{k=1}^{N}\mu_k(\lambda_k+i\nu_k) \int_x^{+\infty}e^{-\nu_ks}\Phi_2(s)ds \int_x^{+\infty}e^{-\nu_ky}\bar{\Phi}_1(y)ds,\notag\\
 \end{eqnarray}
and
 \begin{eqnarray}\label{op_Omega2}
 \int_x^{+\infty}(\Omega_x\Phi)_2\bar{\Phi}_2 &=& -\sum_{k=1}^{N}\mu_k|\int_x^{+\infty}e^{-\nu_ky}\Phi_2(y)dy|^2 \notag\\ 
 &&-\sqrt{2}\sum_{k=1}^{N}\mu_k(\lambda_k-i\nu_k) \int_x^{+\infty}e^{-\nu_ks}\Phi_1(s)ds \int_x^{+\infty}e^{-\nu_ky}\bar{\Phi}_2(y)ds.\notag\\
 \end{eqnarray}
 Combining (\ref{op_Omega1}) and (\ref{op_Omega2}), we obtain 
 \begin{eqnarray}\label{op_Omega3}
 -\int_x^{+\infty}\Omega_x\Phi\cdot\bar{\Phi} &=& \sum_{k=1}^{N}\mu_k|\int_x^{+\infty}e^{-\nu_ky}\Phi_1(y)dy|^2
 + \sum_{k=1}^{N}\mu_k|\int_x^{+\infty}e^{-\nu_ky}\Phi_2(y)dy|^2 \notag\\
 && +2\sqrt{2}\sum_{k=1}^{N}\mu_k\lambda_k Re\left( \int_x^{+\infty}e^{-\nu_ks}\Phi_2(s)ds \int_x^{+\infty}e^{-\nu_ky}\bar{\Phi}_1(y)ds\right)\notag\\
 && -  2\sqrt{2}\sum_{k=1}^{N}\mu_k\nu_k Im\left( \int_x^{+\infty}e^{-\nu_ks}\Phi_2(s)ds \int_x^{+\infty}e^{-\nu_ky}\bar{\Phi}_1(y)ds\right).
 \end{eqnarray}
Since $\mu_k$ are assumed negative,  the sum of the last two terms of the right-hand side of (\ref{op_Omega3}) is upper-bounded by 
 \begin{eqnarray*}
-2\sqrt{2}\sum_{k=1}^{N}\mu_k\sqrt{\lambda^2_k+\nu^2_k}\left|\int_x^{+\infty}e^{-\nu_ks}\Phi_1(s)ds 
\int_x^{+\infty}e^{-\nu_ky}\bar{\Phi}_2(y)dy\right|\\
\leq-\sum_{k=1}^{N}\mu_k\left(|\int_x^{+\infty}e^{-\nu_ky}\Phi_1(y)dy|^2 +
|\int_x^{+\infty}e^{-\nu_ky}\Phi_2(y)dy|^2\right). 
 \end{eqnarray*}
Then 

$$\int_x^{+\infty}\Omega_x\Phi\cdot\bar{\Phi} \geq 0.$$
\end{proof}

The operator $2\sqrt{2}I_d+\Omega_x\in \mathcal{L}(L^2_y(y\geq x))$ is coercive independently of $x.$ Namely, we have 
\begin{equation}\label{Op_Gamma_sp}
\int_x^{+\infty}(2\sqrt{2}I_d +\Omega_x)\Phi\cdot\bar{\Phi} \geq 2\sqrt{2}\|\Phi\|_{L^2}, \quad \forall \Phi \in (L^2_y(y\geq x))^2,\quad \forall x\in \mathbb{R},
\end{equation}
which implies that it is an isomorphism on $(L^2_y(y\geq x))^2$ and that $|\!|\!|(2\sqrt{2}I_d+\Omega_x)^{-1}|\!|\!|_{\mathcal{L}(L^2)} \leq \frac{1}{2\sqrt{2}}.$ 
This allows us to establish the following corollary:
\begin{corollary}\label{Omega_iso}
For every $k\in \mathbb{N},$ the operator $2\sqrt{2}I_d+\Omega_x$ is an isomorphism of\\ $\mathcal{L}(H^k_y(y\geq x)).$ Moreover, we have
$$|\!|\!|(2\sqrt{2}I_d+\Omega_x)^{-1}|\!|\!|_{\mathcal{L}(H^k)} \leq \frac{1}{2\sqrt{2}}.$$
\end{corollary}
\begin{proof}
Let $\Phi \in (H^k_y(y\geq x))^2.$ We know that there exists a unique $\Theta \in (L^2_y(y\geq x))^2$ such that 
$$(2\sqrt{2}I_d+\Omega_x)\Theta = \Phi.$$
In view of the definition of $\Omega_x,$  we have $\Omega_x\Theta \in (H^k_y(y\geq x))^2,$ hence 
$$\Theta = \frac{1}{2\sqrt{2}}(\Omega_x\Theta-\Phi) \in (H^k_y(y\geq x))^2.$$ 
This proves that $2\sqrt{2}I_d+\Omega_x: (H^k_y(y\geq x))^2\rightarrow (H^k_y(y\geq x))^2$ is invertible, hence $2\sqrt{2}$ is not an eigenvalue for 
 $-\Omega_x \in \mathcal{L}((H^k_y(y\geq x))^2).$ This operator is compact, since it is a finite-rank one.  
Then in view of Fredholm alternative, $2\sqrt{2}$ belongs to the resolvent of $-\Omega_x$. This means that $2\sqrt{2}I_d+\Omega_x $ is 
an isomorphism on $(H^k_y(y\geq x))^2.$
On the other hand, the point spectrum of $2\sqrt{2}I_d+\Omega_x$ relatively to the space $(H^k_y(y\geq x))^2$ is 
included in the point spectrum relatively to $(L^2_y(y\geq x))^2$, and similarly for $(2\sqrt{2}I_d+\Omega_x)^{-1}.$ 
Since $\Omega_x \in \mathcal{L}((H^k_y(y\geq x))^2)$
is compact, the spectrum of $2\sqrt{2}I_d+\Omega_x$ contains only eigenvalues. We deduce therefore that 
$$|\!|\!|(2\sqrt{2}I_d+\Omega_x)^{-1}|\!|\!|_{\mathcal{L}(H^k)} \leq |\!|\!|(2\sqrt{2}I_d+\Omega_x)^{-1}|\!|\!|_{\mathcal{L}(L^2)}\leq\frac{1}{2\sqrt{2}}.$$
\end{proof}

Equation (\ref{wp}) is equivalent to  
\begin{equation}\label{wpequiv}
 \Psi^r(x,y) = (2\sqrt{2}I_d+\Omega_x)^{-1} \mathcal{F}\Psi^r(x,y).
\end{equation}
For  $x\in \R$, we use the fixed point argument to prove that
(\ref{wpequiv}) has a unique solution  $\Psi^r(x,.) \in (H^n_y(y\geq x))^2$. To this end, the following lemma will be needed:
 \begin{lemma}\label{Op_f_lip}
  Let $x\in \mathbb{R}.$ The function $\mathcal{F}$ defined above is Lipschitz continuous on $(H^n_{y}(y\geq x))^2$ and there exists 
  a constant $C>0$ depending on $\beta$ such that for all $\Psi^r(x,.),\Psi^l(x,.) \in (H^n_{y}(y\geq x))^2,$ we have 
  $$\|(\mathcal{F}\Psi^r-\mathcal{F}\Psi^l)(x,.)\|_{(H^n_{y}(y\geq x))^2}\leq C\|\Psi^r(x,.)-\Psi^l(x,.)\|_{(H^n_{y}(y\geq x))^2}.$$
 \end{lemma}
\begin{proof}
Let $\Psi^r(x,.),\Psi^l(x,.) \in (H^n_{y}(y\geq x))^2.$ For every $k\leq n,$ we have  
\begin{eqnarray}\label{wp6}
 \partial_y^k(\mathcal{F}\Psi^r-\mathcal{F}\Psi^l)_1(x,y)&=& -\int_x^{+\infty}\sqrt{2}(\Psi^r_2-\Psi^l_2)(x,s)(F_{12}-iF'_{22})^{(k)}(s+y)ds\notag\\
 && + \int_x^{+\infty}(\Psi^r_1-\Psi^l_1)(x,s)F^{(k)}_{22}(s+y)ds,
\end{eqnarray}
\begin{eqnarray}\label{wp7}
 \partial_y^k(\mathcal{F}\Psi^r-\mathcal{F}\Psi^l)_2(x,y)&=& -\int_x^{+\infty}\sqrt{2}(\Psi^r_1-\Psi^l_1)(x,s)(F_{12}+iF'_{22})^{(k)}(s+y)ds\notag\\
 && + \int_x^{+\infty}(\Psi^r_2-\Psi^l_2)(x,s)F^{(k)}_{22}(s+y)ds,
\end{eqnarray}
where 
\begin{equation*}
 (F_{12}-iF'_{22})^{(k)}(s+y) =\frac{1}{2\pi}\int_{\mathbb{R}}(i\xi)^k\left((1+\frac{\xi}{\lambda})\beta(\lambda)
 +(1-\frac{\xi}{\lambda})\beta(-\lambda)\right)e^{i\xi(s+y)}d\xi,
\end{equation*}
$$F^{(k)}_{22}(s+y) =\frac{1}{2\pi}\int_{\mathbb{R}}\frac{(i\xi)^k}{\lambda}(\beta(\lambda)-\beta(-\lambda))e^{i\xi(s+y)}d\xi.$$
Denote
$$d_k(-\xi)= (i\xi)^k\left((1+\frac{\xi}{\lambda})\beta(\lambda)
 +(1-\frac{\xi}{\lambda})\beta(-\lambda)\right),$$
$$e_k(-\xi)= \frac{(i\xi)^k}{\lambda}(\beta(\lambda)-\beta(-\lambda)).$$
Since  $\lambda\mapsto \lambda^{n+2}\beta(\lambda)$ is assumed to be in $L^\infty([\frac{1}{\sqrt{2}},+\infty[),$ there exists $C>0$ such that  
$$|\xi^k\lambda^2\beta(\lambda(\xi))| \leq \lambda^{k+2}|\beta(\lambda)|\leq C , k\in \{0,...,n\}.$$
Then $d_k, e_k \in L^1(\mathbb{R})\cap L^\infty(\mathbb{R})$  and  (\ref{wp6}) can be rewritten as follows 
\begin{eqnarray}\label{wp66}   
\partial_y^k (\mathcal{F}\Psi^r-\mathcal{F}\Psi^l)_1(x,y)&=& 
 \int_{\mathbb{R}}\omega^1_x(s)\widehat{e_k}(s+y)ds-\int_{\mathbb{R}}\omega^2_x(s)\widehat{d_k}(s+y)ds\notag\\
 &=& (\omega^1_x\star\widehat{e_k}(-\cdot)) (-y) -(\omega^2_x\star\widehat{d_k}(-\cdot)) (-y),\notag\\
\end{eqnarray}
with $\omega^1_x(s) =1_{]x,+\infty[}(\Psi^r_1-\Psi^l_1)(x,s),$  $\omega^2_x(s) =\sqrt{2}1_{]x,+\infty[}(\Psi^r_2-\Psi^l_2)(x,s),$ where 
 $\widehat{d_k}, \widehat{e_k}$ are the Fourier transforms of $d_k,c_k$ respectively. Since the functions 
$\widehat{d_k}, \widehat{e_k},$ $\omega^1_x$ and $\omega^2_x$ belongs to $L^2(\mathbb{R}),$ we have
$$\omega^1_x\star\widehat{e_k}(-\cdot) = \mathbf{F}^{-1}[\widehat{\omega^1_x}d_k],\quad \omega^2_x\star\widehat{e_k}(-\cdot) = \mathbf{F}^{-1}[\widehat{\omega^2_x}e_k],$$
where $\mathbf{F}^{-1}$ is the inverse Fourier transform. This implies that the left hand side members of previous equations belong to $L^2(\mathbb{R})$
and that $\partial_y^k(\mathcal{F}\Psi^r-\mathcal{F}\Psi^l)(x,.) \in (L^2_{y}(y\geq x))^2.$  
Applying the Fourier transform on  (\ref{wp66}) and taking the $L^2(\mathbb{R})$ norm, we obtain 
\begin{eqnarray*}
\|\partial_y^k(\mathcal{F}\Psi^r-\mathcal{F}\Psi^l)_1(x,.)\|_{L^2_{y}(y\geq x)}&\leq& \sqrt{2}\|d_k\|_{L^\infty(\mathbb{R})}\|(\Psi^r_2-\Psi^l_2)(x,.)\|_{L^2_{y}(y\geq x)}\\
&& + \|e_k\|_{L^\infty(\mathbb{R})}\|(\Psi^r_1-\Psi^l_1)(x,.)\|_{L^2_{y}(y\geq x)}\\
&\leq& 2\sqrt{2}\|\lambda^k\beta(.)\|_{L^\infty([\frac{1}{\sqrt{2}},+\infty[)}\\
&& \times (\|(\Psi^r_1-\Psi^l_1)(x,.)\|_{L^2_{y}(y\geq x)}\\
&& \quad +\|(\Psi^r_2-\Psi^l_2)(x,.)\|_{L^2_{y}(y\geq x)}).
\end{eqnarray*}
A similar argument yields the same inequality from (\ref{wp7}). This completes the proof.  
\end{proof}

Consequently, if $\|\lambda^n\beta(.)\|_{L^\infty[\frac{1}{\sqrt{2}},+\infty[}$ is small enough, the operator $ (2\sqrt{2}I_d+\Omega_x)^{-1}\mathcal{F}$ 
has a unique fixed point $\Psi^r(x,.)\in (H^n_y(y\geq x))^2$ which is the solution of (\ref{wp}).
\subsection{Variation for Proposition \ref{wp_Marchenko_2} : addition of time dependence}
Assume that 
$$\beta(t,\lambda)= c(\lambda)\exp(-4i\lambda\xi t),\quad \mu_k(t)= \mu^0_k\exp(4\lambda_k\nu_k t),\quad 1\leq k\leq N.$$
In the case $c\equiv  0,$ we deduce, following the same steps followed in Section \ref{sec_N_solit}, that there exist two families 
$\{f_k\}, \{g_k\}$ of functions in the space $\mathcal{C}^\infty(\mathbb{R}^2)$ such that for all $(t,x)\in \mathbb{R}^2,$ the function
$$y\mapsto \Upsilon(t,x,y) = \left(\sum_{k=1}^Nf_k(t,x)e^{-\nu_k y},\sum_{k=1}^Ng_k(t,x)e^{-\nu_k y}\right)^t$$
is the unique solution of the corresponding Marchenko system. The dependence on $t$ of the two operators $\Omega_x$ and $\mathcal{F}$
does not change the previous results. More specifically, the operator $\Omega_x$ stays non-negative independently of $(t,x)\in \mathbb{R}^2$.
This allows us to establish a similar corollary to Corollary \ref{Omega_iso}. On the other hand, we have 
 $$\|\xi^k\beta(t,\lambda(.))\|_{L^\infty(\mathbb{R})}= \|\xi^kc(\lambda(.))\|_{L^\infty(\mathbb{R})}, k\leq n.$$
Hence, if $\|\lambda^nc(.)\|_{L^\infty([\frac{1}{\sqrt{2}},+\infty[)}$ is small enough, for all $(t,x)\in \mathbb{R}^2,$ 
the operator $ (2\sqrt{2}I_d+\Omega_x)^{-1}\mathcal{F}$ has a unique fixed point 
$\Psi^r(t,x,.)\in (H^n_y(y\geq x))^2$ which is the solution of (\ref{wp}).

\subsection{Regularity with respect to $x$ and $t$ of Marchenko system solutions}\label{sec_reg}
To give sense to such regularity, the following reformulation 
$$\underline{\Phi}(x,p)= \Phi(x,x+p),\quad (x,p)\in \mathbb{R}\times\mathbb{R}^+,$$ 
which makes the domain of variables fixed, will be useful.  Equation (\ref{wp}) can be rewritten as follows
\begin{equation}\label{00wp}
(2\sqrt{2}I_d+\underline{\Omega}_x) \underline{\Psi}^r =  \mathcal{T}_x\underline{\Psi}^r + \mathcal{T}_x\underline{\Upsilon} + F_x,
\end{equation} 
where $\mathcal{T}_x$ is the linear part of the operator  $\mathcal{F}$.  Namely, 
$$\mathcal{T}_x(t)\Phi(p)  = \int_0^{+\infty}\widehat{T}(t,s+2x+p)\Phi(s)ds,\quad
T= \left(\begin{matrix}e_0& -\sqrt{2}d_0\\ -\sqrt{2}\bar{d_0}& e_0\end{matrix}\right),$$
$$\underline{\Omega}_x(t)\Phi(p)  = \int_0^{+\infty}\Omega(t,s+2x+p)\Phi(s)ds,\quad
\Omega= -\left(\begin{matrix} F_{21}& \sqrt{2}(F_{11}-iF'_{21})\\ \sqrt{2}(F_{11}+iF'_{21})& F_{21}\end{matrix}\right),$$
and $$F_x = \left(\begin{matrix} F_{22}(2x+\cdot)\\ \sqrt{2}(F_{12}+iF'_{22})(2x+\cdot) \end{matrix}\right).$$
The operator $(2\sqrt{2}I_d+\underline{\Omega}_x(t))$ has the same properties as   $(2\sqrt{2}I_d+\Omega_x(t)),$ 
i. e. it is coercive independently of $t,x$ on the space $(L^2(\mathbb{R}^+))^2.$ We set 
\begin{eqnarray*}
 \underline{\Psi}^r &=& (2\sqrt{2}I_d+\underline{\Omega}_x)^{-1}\mathcal{T}_x\underline{\Psi}^r + (2\sqrt{2}I_d+\underline{\Omega}_x)^{-1}(\mathcal{T}_x\underline{\Upsilon} + F_x)\\
&:=& \mathcal{M}_x\underline{\Psi}^r + R_x. 
\end{eqnarray*}
 
\begin{lemma}\label{op_cont}
Assume that $n\geq 3.$  Let $k$ be an integer such that $ k< \frac{n}{2}.$ 
The operator  $\mathcal{M}_x(t) \in \mathcal{L}((H^{n-2k}({\mathbb{R}^+}))^2)$ is contraction independently of 
$t,x.$ Moreover, The application
$$\mathbb{R}^2\rightarrow  \mathcal{L}((H^{n-2k}({\mathbb{R}^+}))^2): (t,x)\mapsto \mathcal{M}_x(t)$$ 
is of class $\mathcal{C}^k(\mathbb{R}^2).$ 
\end{lemma}
\begin{proof}
The contraction property of  $\mathcal{M}_x(t)$ is a result of the fact that  $(2\sqrt{2}I_d+\Omega_x(t))^{-1}\mathcal{F}$ is contraction 
 independently of $(t,x).$ The two applications $(t,x)\mapsto \mathcal{T}_x(t)$ and $(t,x)\mapsto 2\sqrt{2}I_d+\underline{\Omega}_x(t)$ 
 are of classes $\mathcal{C}^k(\mathbb{R})$ and $\mathcal{C}^\infty(\mathbb{R})$ respectively.
In fact, we have $\widehat{T}\in H^n(\mathbb{R})$. Then there exists $C>0$ such that for all $t\in \mathbb{R},$ $\Phi \in (H^1({\mathbb{R}^+}))^2$ 
and $j\in \{0,...,n-1\}$, the $L^2(\mathbb{R}^+)$ norm of the term
$$\left[\left(\frac{1}{h}\left(\tau^x_{2h}\widehat{T}(-\cdot)-\widehat{T}(-\cdot)\right)-\widehat{T}'(-\cdot)\right)^{(j)}(t)\star1_{\mathbb{R}^+}\Phi\right](-2x-\cdot)$$  
is upper-bounded by  $$C\left\|\left(\frac{1}{h}\left(\tau^x_{2h}\widehat{T}-\widehat{T}\right)-\widehat{T}'\right)^{(j)}(t)\right\|_{L^2(\mathbb{R})}\|\Phi\|_{H^1(\mathbb{R}^+)},$$
where $\tau^x_{2h}$ is the  translation with respect to $x.$  Hence, the fact that
$$\lim_{h\rightarrow 0}\left\|\left(\frac{1}{h}\left(\tau^x_{2h}\widehat{T}-\widehat{T}\right)-\widehat{T}'\right)^{(j)}(t)\right\|_{L^2(\mathbb{R})} =0,$$
proves that the application $\mathbb{R}\rightarrow \mathcal{L}(H^{n-1}): x\mapsto \mathcal{T}_x(t)$ is differentiable. 
On the other hand, by dominated convergence theorem and the continuity of Fourier transform on  $L^2(\mathbb{R})$, 
for every $j\in \{0,...,n-2\}$ we have  
$$\lim_{h\rightarrow 0}\left\|\left(\frac{1}{h}\left(\tau^t_{h}\widehat{T}-\widehat{T}\right)-\partial_t\widehat{T}\right)^{(j)}(t)\right\|_{L^2(\mathbb{R})} =0.$$
This proves that $\mathbb{R}\rightarrow \mathcal{L}(H^{n-2}): t\mapsto \mathcal{T}_x(t)$ is differentiable. 
Similar arguments prove that the partial derivatives 
 $\partial_t^j\partial_x^l\mathcal{T}_x(t)$ exist in $\mathcal{L}((H^{n-2j-l}(\mathbb{R}))^2).$ On the one hand, we have
$\partial^j_t \beta(t,\lambda)= (-4i\lambda\xi)^j\beta(t,\lambda)$ et  $|\lambda\xi|^j \leq \langle\xi\rangle^{2j}.$ 
and on the other hand, the partial derivative $\partial_x^l$ of the kernel $\widehat{T}$ in $\mathcal{T}_x$ involves the power $|\xi|^j.$
Hence, since $\widehat{T} \in H^n(\mathbb{R}),$
the application $(t,x)\mapsto \mathcal{T}_x(t)$ is of class $\mathcal{C}^k(\mathbb{R}^2,(H^{n-2k}({\mathbb{R}^+}))^2).$ Moreover,  
the application $ U\mapsto U^{-1} \in\mathcal{GL}((H^{n-2k}({\mathbb{R}^+}))^2)$ is of class $\mathcal{C}^\infty$. Consequently,    
$$\mathbb{R}^2\rightarrow  \mathcal{L}((H^{n-2k}({\mathbb{R}^+}))^2): (t,x)\mapsto \mathcal{M}_x= (2\sqrt{2}I_d+\underline{\Omega}_x)^{-1}\mathcal{T}_x$$ 
is of class $\mathcal{C}^k(\mathbb{R}^2).$ 
\end{proof}

For $(n,k)\in \mathbb{N}^*\times\mathbb{N}$ such that $n\geq 3$ and $k<\frac{n}{2},$ we define the function \\
$\mathcal{K}:\mathbb{R}^2\times(H^{n-2k}(\mathbb{R}^+))^2 \rightarrow (H^{n-2k}(\mathbb{R}^+))^2 $  as follows: 
$$\mathcal{K}(t,x,\Phi) = \Phi -\mathcal{M}_x(t)\Phi - R_x,\quad (x,\Phi) \in \mathbb{R}\times(H^{n-2k}(\mathbb{R}^+))^2.$$ 
In view of Lemma \ref{op_cont}, $\mathcal{K}$ is of class $\mathcal{C}^k.$ Finally, for all $(t,x)\in \mathbb{R}^2,$ we have  
\begin{enumerate}
 \item $\mathcal{K}(t,x,\underline{\Psi}^r(t,x,.)) = 0.$
 \item $\partial_\Phi\mathcal{K}(t,x,\underline{\Psi}^r(t,x,.)) = I_d - \mathcal{M}_x(t).$
\end{enumerate}
In view of Lemma \ref{op_cont}, $\mathcal{M}_x(t)$ is a contraction, $|\!|\!|\mathcal{M}_x(t)|\!|\!|_{\mathcal{L}(H^{n-2k})}<1.$ This implies that  
 $\partial_\Phi\mathcal{K}(t,x,\underline{\Psi}^r(t,x,.))$ is an automorphism. Consequently, 
both the implicit function theorem and the uniqueness of $\underline{\Psi}^r(t,x,.)$ imply that
$ \underline{\Psi}^r  \in \mathcal{C}^k(\mathbb{R}^2,(H^{n-2k}(\mathbb{R}^+))^2).$
For fixed $t\in \R$, a similar argument proves that if $n \in \mathbb{N}^*,$ then for every integer $k$ such that $k<n$,  
 $\underline{\Psi}^r(t,.,.) \in \mathcal{C}^k(\mathbb{R}, (H^{n-k}(\mathbb{R}^+))^2)$.

%%%%%%%%%%%%%%%%%%%%%%% %%%%%%%%%%%%%%%%%%%%%%%%%%%%%%%%%%%%%%%%% 

\subsection{Construction of the Gross-Pitaevskii solution}
In this section it will be proved that the function
$$u(t,x)= 1+2\sqrt{2}i\overline{\Psi}_2(t,x,x)= 1+2\sqrt{2}i\underline{\overline{\Psi}}_2(t,x,0)$$
 is a  solution of Gross-Pitaevskii equation. To this end, the following lemma will be useful
 \begin{lemma}\label{wp_lemme}
 Let $(\Psi_1,\Psi_2)^t$ the solution of Marchenko system under the assumptions of Proposition \ref{wp_Marchenko_2} with $n\geq 3.$ Denote 
  $$\Psi = \left(\begin{array}{lr}
                  \Psi_1& \Psi_2\\
                  \overline{\Psi}_2& \overline{\Psi}_1
                 \end{array}
\right).$$
Let  $(\lambda,\xi) \in \mathbb{C}^2,$ such that $\lambda^2-\xi^2 = \frac{1}{2}$ and $Im(\xi) <0.$
Let $X(x,\lambda)=e^{-i\xi x}\left(\begin{matrix}1\\ \sqrt{2}(\lambda - \xi)\\ \end{matrix}\right)$. We set 
\begin{eqnarray*}
\psi(t,x,\lambda) &=& X(x,\lambda) -\int_x^{+\infty}\Psi(t,x,s)X(s,\lambda)ds\\
&=& \left(Id -\int_0^{+\infty}\underline{\Psi}(t,x,p)e^{-i\xi p}dp\right)X(x,\lambda)ds,
\end{eqnarray*}
$$\chi_1 = (\sqrt{3}-1)^{\frac{1}{2}}e^{i\frac{E}{2}x}\psi_1,\quad \chi_2 = (\sqrt{3}+1)^{\frac{1}{2}}e^{i\frac{E}{2}x}\psi_2,\quad 
E = \frac{2}{\sqrt{3}}\lambda.$$
Then $\psi(t,.,\lambda)\in \mathcal{C}^2(\mathbb{R})$ is a solution of 
\begin{equation}
\label{Zakharov02}
iM\partial_xV+QV-\lambda V=0,
\end{equation}
with $$M=
\left(\begin{matrix}
1 & 0\\
0   &-1\\
\end{matrix}\right) ,\quad
Q(x)=
\left(\begin{matrix}
0 & \bar{q}(x)\\
q(x)   &0\\
\end{matrix}\right),\quad \text{and}\quad q= \frac{\sqrt{2}}{2}u.
$$
Moreover, we have  
\begin{equation}\label{wp8}
 \partial_t\chi -B\chi= i\sqrt{3}(\frac{E}{2}-\xi)^2\chi, 
\end{equation}
with $\chi = \left(\begin{array}{lr}\chi_1\\ \chi_2\end{array}\right),$  
$B=
-\sqrt{3}i\left(\begin{matrix}
1 & 0\\
0   &1\\
\end{matrix}\right)\partial^2_x
+
\left(\begin{matrix}
\frac{|u|^2-1}{\sqrt{3}+1}i & -\partial_x\bar{u}\\
\partial_xu   &\frac{|u|^2-1}{\sqrt{3}-1}i\\
\end{matrix}\right).
$ 
 \end{lemma}
\begin{proof}
The proof depends essentially on the well-posedness of Marchenko system and the choice of the function 
$$\beta(t,\lambda)= c(\lambda)\exp(-4i\lambda\xi t).$$ 
Since $n\geq 3,$ we have the following regularity    
$\underline{\Psi}= \left(\begin{matrix}\underline{\Psi}_1\\ \underline{\Psi}_2\end{matrix}\right) \in \mathcal{C}^1(\mathbb{R}^2,H^1(\mathbb{R}^+)^2)$
and $\underline{\Psi}(t,.,.)\in \mathcal{C}^2(\mathbb{R},H^1(\mathbb{R}^+)^2 )$ for all $t\in \mathbb{R}$. 
To prove that  $\psi$ is a  solution of (\ref{Zakharov02}),  it suffices to prove that $\Psi$ satisfies the linear system 
\begin{equation}\label{Zakharov05}
 iM\partial_x\Psi+i\partial_y\Psi M - \Psi Q_0+ Q(x)\Psi = 0,\quad Q_0 =\left(\begin{matrix} 0 & \frac{1}{\sqrt{2}}\\ \frac{1}{\sqrt{2}}   &0\\ \end{matrix}\right).
\end{equation}
In fact, replacing $Q$ by $Q_0,$ $X$ satisfies the system (\ref{Zakharov02}). Hence, we have 
\begin{eqnarray*}
 iM\partial_x\psi+Q\psi-\lambda \psi &=&  (Q-Q_0)X -(iM\partial_x+Q-\lambda Id)\int_x^{+\infty}\Psi(x,s)X(s,\lambda)ds \\
&=& (Q-Q_0)X +iM\Psi(x,x)X \\
&& -\int_x^{+\infty}(-i\partial_s\Psi(x,s)M+\Psi(x,s)Q_0-\lambda \Psi(x,s))X(s,\lambda)ds \\ 
&=& (Q-Q_0 +iM\Psi(x,x) -i\Psi(x,x)M)X\\
&& -\int_x^{+\infty}\left(\xi \Psi(x,s)M+\Psi(x,s)Q_0-\lambda \Psi(x,s)\right)X(s,\lambda)ds. \\ 
\end{eqnarray*}
A simple calculation proves that  $\left(\xi \Psi(x,s)M+\Psi(x,s)Q_0-\lambda \Psi(x,s)\right)X(s,\lambda) = 0,$ 
and since  $u(x)= 1+2\sqrt{2}i\overline{\Psi}(x,x),$ we have  
$$Q-Q_0 +iM\Psi(x,x) -i\Psi(x,x)M=  0.$$ 
To verify that $\Psi$ satisfies (\ref{Zakharov05}), we set  
$$C_1= (\partial_x+\partial_y)\Psi_1-i\left(-\frac{i}{\sqrt{2}}\Psi_2+\bar{q}\overline{\Psi}_2\right),$$
$$C_2= (\partial_x-\partial_y)\Psi_2-i\left(-\frac{i}{\sqrt{2}}\Psi_1+\bar{q}\overline{\Psi}_1\right).$$
Then by direct calculation we find that $p\mapsto C(x,p) =(C_1(x,x+p), C_2(x,x+p))^t \in (L^2(\mathbb{R}^+))^2$ satisfies the homogeneous
Marchenko system 
$$(2\sqrt{2}I_d+ \underline{\Omega}_x)C = \mathcal{T}_x C ,\quad x\in \mathbb{R},$$
which has a unique trivial solution $C=0.$ 
This proves that $\Psi$ satisfies the linear system (\ref{Zakharov05}). To prove (\ref{wp8}), denote  $h= e^{i\xi x}\psi.$ Then (\ref{wp8}) is equivalent to
\begin{equation}\label{wp9}
 \partial_th = g,
\end{equation}
with  
$$
 g= \left(\begin{array}{lr} -i\frac{|u|^2-1}{\sqrt{3}+1}& \frac{\sqrt{3}+1}{\sqrt{2}}\partial_x\bar{u}\\ -\frac{\sqrt{3}-1}{\sqrt{2}}\partial_xu &-i\frac{|u|^2-1}{\sqrt{3}-1} \end{array}\right)h 
+(2\sqrt{3}\xi-2\lambda)\partial_xh+i\sqrt{3}\partial^2_xh.
$$
Since $\psi$ is a solution of (\ref{Zakharov02}), we have  
$$\lambda \partial_xh = \left(\begin{array}{lr} 0&\partial_x\bar{q}\\ \partial_xq & 0  \end{array}\right)h +\left(\begin{array}{lr} \xi&\bar{q}\\ q& -\xi \end{array}\right)\partial_xh 
+\left(\begin{array}{lr} i&0\\ 0& -1\end{array}\right)\partial^2_xh. 
 $$
Substituting into the expression of $g,$ we find that 
\begin{eqnarray}\label{wp10}
 g&= &\left(\begin{array}{lr} -i\frac{|u|^2-1}{\sqrt{3}+1}& \frac{\sqrt{3}-1}{\sqrt{2}}\partial_x\bar{u}\\ -\frac{\sqrt{3}+1}{\sqrt{2}}\partial_xu &-i\frac{|u|^2-1}{\sqrt{3}-1} \end{array}\right)h
+ 2\left(\begin{array}{lr} (\sqrt{3}-1)\xi& -\bar{q}\\ -q& (\sqrt{3}+1)\xi\end{array}\right)\partial_x h \notag\\
&&+i\left(\begin{array}{lr} \sqrt{3}-2 &0\\ 0& \sqrt{3}+2\end{array}\right)\partial^2_xh.
\end{eqnarray}
Next, note that the definition of $\psi$ implies that
$$h(t,x,\lambda) = \left(\begin{matrix} 1\\ \sqrt{2}(\lambda-\xi)\\ \end{matrix}\right)-\int_0^{+\infty}\Psi(t,x,x+s)X(s,\lambda)ds.$$ 
Hence, Substituting into (\ref{wp10}), integrating by parts and  using (\ref{Zakharov05}), we obtain 
$$g(t,x,\lambda) = \int_0^{+\infty}G(t,x,x+s)X^+_1(s,\lambda)ds,$$
with 
$$G=  \left(\begin{matrix} G_1&G_2\\  \overline{G}_2&\overline{G}_1\end{matrix} \right)=
                  \left(\begin{matrix}
                i(\partial_x+\partial_y)^2\Psi_1+2\bar{q}(\partial_x+\partial_y)\overline{\Psi}_2&   i(\partial_x+\partial_y)^2\Psi_2+2\bar{q}(\partial_x+\partial_y)\overline{\Psi}_1\\
                -i(\partial_x+\partial_y)^2\overline{\Psi}_2+2q(\partial_x+\partial_y)\Psi_1&   -i(\partial_x+\partial_y)^2\overline{\Psi}_1+2q(\partial_x+\partial_y)\Psi_2\\
                \end{matrix} \right).
$$
On the other hand, we have 
$$\partial_t h(t,x,\lambda) = -\int_0^{+\infty}\partial_t\Psi(t,x,x+s)X(s,\lambda)ds.$$
Then to prove (\ref{wp9}), it suffices to prove that for every $(t,x)\in \mathbb{R}^2$
\begin{equation}
 \partial_t\Psi(t,x,.) = -G(t,x,.).
\end{equation}
To this end, we verify that for every $(t,x)\in\mathbb{R}^2,$ 
the two functions $\partial_t\left(\begin{matrix}\Psi_1\\ \Psi_2\end{matrix}\right)$ and  $\left(\begin{matrix}G_1\\G_2\end{matrix}\right)$ 
satisfy the same Marchenko equations in the space $(L_y(y\geq x))^2$ by using the relationships  
$$\partial_tF_1(t,z)= -2\partial_zF_2(t,z)+ 2\partial^3_zF_2(t,z),$$
$$\partial_tF_2(t,z)= -4\partial_zF_1(t,z),$$
which are provided by the definition of $F_1$ and $F_2.$ Differentiating (\ref{marchenko1}) (in its version depending on $t$) with respect to $t,$  
we have for every $y\geq x$
\begin{eqnarray*}
&&2\sqrt{2}\partial_t\Psi_1(t,x,y)=D_1(t,x,y)\\
 &&-\int_{x}^{+\infty}\sqrt{2}\Psi_2(t,x,s)(F_1(t,s+y)-iF'_2(t,s+y))+\Psi_1(t,x,s)F_2(t,s+y)ds,\\
&&2\sqrt{2}\partial_t\Psi_2(t,x,y)=D_2(t,x,y)\\
&&-\int_{x}^{+\infty}\sqrt{2}\Psi_1(t,x,s)(F_1(t,s+y)+iF'_2(t,s+y))+\Psi_2(t,x,s)F_2(t,s+y)ds,\\
\end{eqnarray*}
with
\begin{eqnarray*}
D_1(t,x,y)&=& -\int_{x}^{+\infty}\sqrt{2}\Psi_2(t,x,s)(4F'''_2(t,s+y)-2F'_2(t,s+y)+4iF'_1(t,s+y))ds\\
 &&+4\int_{x}^{+\infty}\Psi_1(t,x,s)F'_1(t,s+y)ds-4F'_1(t,x+y)),\\
D_2(t,x,y)&=& 4\int_{x}^{+\infty}\Psi_2(t,x,s)F'_1(t,s+y)ds\\
&&-\int_{x}^{+\infty}\sqrt{2}\Psi_1(t,x,s)(4F'''_2(t,s+y)-2F'_2(t,s+y)-4iF'_1(t,s+y))ds\\
&& +\sqrt{2}(4F'''_2(t,x+y)-2F'_2(t,x+y)-4iF'_1(t,x+y)).
\end{eqnarray*}
On the other hand, differentiating (\ref{marchenko1}) with respect to $x$ and with respect to $y,$ integrating by parts and using (\ref{Zakharov05}), we obtain
\begin{eqnarray*}
&&2\sqrt{2}\partial_tG_1(t,x,y)=D_1(t,x,y)\\
 &&-\int_{x}^{+\infty}\sqrt{2}G_2(t,x,s)(F_1(t,s+y)-iF'_2(t,s+y))+G_1(t,x,s)F_2(t,s+y)ds,\\
&&2\sqrt{2}\partial_tG_2(t,x,y)=D_2(t,x,y)\\
&&-\int_{x}^{+\infty}\sqrt{2}G_1(t,x,s)(F_1(t,s+y)+iF'_2(t,s+y))+G_2(t,x,s)F_2(t,s+y)ds.\\
\end{eqnarray*} 
Consequently, for every $(x,y)\in \mathbb{R}^2,$ the two functions $\partial_t\left(\begin{matrix}\Psi_1(t,x,.)\\ \Psi_2(t,x,.)\end{matrix}\right)$ and
$\left(\begin{matrix}G_1(t,x,.)\\G_2(t,x,.)\end{matrix}\right)$  satisfy the same Marchenko equations in the space $(L_y(y\geq x))^2$.
Thus, the conclusion follows from the uniqueness.
\end{proof}

Since $\psi$ is a solution of (\ref{Zakharov02}), $\chi$ is a solution of 
$$L\chi = E\chi,$$
with $L=
\left(\begin{matrix}
(1+\sqrt{3})i\partial_x & \bar{u} \\
u   &(1-\sqrt{3})i\partial_x
\end{matrix}\right),$ and in view of\eqref{wp8}, $\partial_t\chi -B\chi$ is also a solution for the same equation, i. e.  
$$L(\partial_t\chi -B\chi)= E(\partial_t\chi -B\chi).$$
On the other hand, we have  
$$\partial_tL\chi + L\partial_t\chi= E\partial_t\chi.$$
Combining the last two equations, we find that 
$$(\partial_tL - [L,B])\chi = 0.$$
Hence, we have $$(i\partial_tu+ \partial^2_xu + (1-|u|^2)u)\left(1-\int_0^{+\infty}\left(\underline{\Psi}_1+\sqrt{2}(\lambda-\xi)\underline{\Psi}_2\right) (x,p)e^{-i\xi p}dp\right) = 0.$$
Let now $\xi = \frac{1}{\sqrt{2}}(s-i)$ and $\lambda = \frac{1}{\sqrt{2}}(f(s)-i\frac{s}{f(s)})$ with $f(s) = \frac{1}{\sqrt{2}}\sqrt{s^2 + \sqrt{s^4 + 4s^2}}.$   
We have
\begin{equation}\label{GP_u}
\int_0^{+\infty}\left(\underline{\Psi}_1+\sqrt{2}(\lambda-\xi)\underline{\Psi}_2\right) (x,p)e^{-i\xi p}dp = 
\hat{f}_1\left(\frac{s}{\sqrt{2}}\right) +\sqrt{2}(\lambda-\xi) \hat{f}_2\left(\frac{s}{\sqrt{2}}\right),
\end{equation}
with $f_1(p) = \underline{\Psi}_1(p)e^{-\frac{p}{\sqrt{2}}}1_{\mathbb{R}^+},  f_2(p) = \underline{\Psi}_2(p)e^{-\frac{p}{\sqrt{2}}}1_{\mathbb{R}^+}.$
Thus, for $s\geq 1,$ the function $s\mapsto\lambda(s)-\xi(s)$ is bounded and the right hand side member of 
(\ref{GP_u}) converges to zero when $s$ tends to $+\infty.$ Then 
for every $x,t \in \mathbb{R},$ there exists $s_0\geq 1$ such that $1-\hat{f}_1\left(\frac{s_0}{\sqrt{2}}\right) -\sqrt{2}(\lambda-\xi) \hat{f}_2\left(\frac{s_0}{\sqrt{2}}\right) \neq 0.$ 
This proves that     
$$i\partial_tu+ \partial^2_xu + (1-|u|^2)u = 0 \quad \text{on} \quad \mathbb{R}^2,$$
and $u$ is therefor a solution of Gross-Pitaevskii equation. 

\subsection{Behavior of the constructed solution when $x\rightarrow \pm\infty$}
%%%%%%%%%%%%%%%%%%%%%%%%%%%%%%%%%%%%%%%%%%%%%%%%%
It was proved before that the solution $\Upsilon$ of Marchenko system corresponding to  $\beta \equiv 0$ is of the form
$$\Upsilon(t,x,y) = \sum_{k=1}^N\left(\begin{matrix}f_k(t,x)\\g_k(t,x)\end{matrix}\right)e^{-\nu_k y}, y\geq x,$$
or 
$$\underline{\Upsilon}(t,x,p) = \sum_{k=1}^N\left(\begin{matrix}f_k(t,x)e^{-\nu_k x}\\g_k(t,x)e^{-\nu_k x}\end{matrix}\right)e^{-\nu_k p}, p\geq 0.$$
We have the following lemma 

\begin{lemma} \label{Upsilon_born}
If  $\nu_k$ are pairwise distinct, the functions   
$$\tilde{g}_k(t,x) = g_k(t,x)e^{-\nu_k x},\quad \tilde{f}_k(t,x) = f_k(t,x)e^{-\nu_k x}$$
are bounded on $\mathbb{R}^2$.
\end{lemma}
\begin{proof}
Denote $\quad a_k(t,x) = -2\sqrt{2}\mu_k^{-1}(t)e^{2\nu_k x}.$
Then developing (\ref{marchenko222}) and (\ref{marchenko333}), we obtain  
\begin{equation*}
 \left\{\begin{array}{lr}
        a_k\tilde{f}_k + 2\sum_{j=1}^N\frac{\tilde{f}_j}{\nu_j+\nu_k} = 1,\quad k\in\{1,...,N\}&\\
        a_k\tilde{g}_k + 2\sum_{j=1}^N\frac{\tilde{g}_j}{\nu_j+\nu_k} = \lambda_k-i\nu_k,\quad k\in\{1,...,N\}.&\\ 
       \end{array}
       \right.
\end{equation*}
We rewrite the previous system in the form of the following two linear systems  
\begin{equation}\label{f_g_born}
 (D_a(t,x)+\Lambda)\tilde{f}(t,x) = K_1,\quad  (D_a(t,x)+\Lambda)\tilde{g}(t,x) = K_2,
\end{equation}
with $\tilde{f}=(\tilde{f}_1,...,\tilde{f}_N)^t,$ $\tilde{g}=(\tilde{g}_1,...,\tilde{g}_N)^t,$  $D_a= diag(a_1,...,a_N),$ $K_1 =(1,...,1)^t$ and\\
$K_2 =(\lambda_1-i\nu_1,...,\lambda_N-i\nu_N)^t.$ The matrix $\Lambda$ is coercive. In fact, for $Z\in \mathbb{C}^N,$ we have 
$$\frac{1}{\nu_j+\nu_k} = \int_0^{+\infty}e^{-(\nu_k+\nu_j)s}ds ,\quad \langle \Lambda Z, Z\rangle = \int_0^{+\infty}|\sum_{j=1}^Nz_je^{-\nu_j s}|^2ds.$$
It follows from the equivalence of the norms on $vect(e^{-\nu_1s},...,e^{-\nu_Ns})$ that there exists $C >0$ such that  
$$\int_0^{+\infty}|\sum_{j=1}^Nz_je^{-\nu_j s}|^2ds \geq C|Z|^2.$$
Thus, since the functions $a_k$ are non-negative, we find, using the $\mathbb{C}^N$ usual scalar product in (\ref{f_g_born}) for 
every $(t,x)\in \mathbb{R}^2$ with $(\tilde{f}_1(t,x),...\tilde{f}_N(t,x))^t$ and $(\tilde{g}_1(t,x),...\tilde{g}_N(t,x))^t$, that 
$$\sum_{k=1}^N \tilde{f}_k(t,x), \sum_{k=1}^N (\lambda_k-i\nu_k)\bar{\tilde{g}}_k(t,x) \in \mathbb{R}^+,$$
and for every $(t,x)\in \mathbb{R}^2$,  
$$C\sum_{k=1}^N |\tilde{f}_k(t,x)|^2 \leq \sum_{k=1}^N \tilde{f}_k(t,x),\quad C\sum_{k=1}^N |\tilde{g}_k(t,x)|^2 \leq \sum_{k=1}^N (\lambda_k-i\nu_k)\bar{\tilde{g}}_k(t,x).$$
This completes the proof.  
\end{proof}
%%%%%%%%%%%%%%%%%%%%%%%%%%%%%%%%%%%%%%%%%%%%%%%%

Denote now $u_N(t,x)$ the solution of (GP) corresponding to $\beta \equiv 0.$ We know that  
$$u_N(t,x)= 1+2\sqrt{2}i\underline{\overline{\Upsilon}}_2(t,x,0).$$
Then 
\begin{equation}\label{u_u_N}
 u(t,x) = u_N(t,x) + 2\sqrt{2}i\underline{\overline{\Psi}}^r_2(t,x,0).
\end{equation}

\begin{lemma}
Under the assumptions of Proposition \ref{wp_Marchenko_2}, for fixed $t\in \mathbb{R}$, we have  
 $$\lim_{x\rightarrow \pm\infty}\underline{\Psi}^r(t,x,0)= 0.$$
\end{lemma}
\begin{proof}
For fixed $t\in \mathbb{R}$, $\underline{\Psi}^r \in \mathcal{C}^b(\mathbb{R}, (H^n(\mathbb{R}^+))^2).$ In fact, we have  
 $$\underline{\Psi}^r = \mathcal{M}_x\underline{\Psi}^r+ R_x,$$
 with $R_x= \mathcal{M}_x\underline{\Upsilon} + (2\sqrt{2}+\underline{\Omega}_x)^{-1}F_x.$
 In view of Lemma \ref{op_cont},  the operator $\mathcal{M}_x\in \mathcal{L}((H^n(\mathbb{R}^+))^2)$ is $\theta-$contraction.  
 Then for $x\in \mathbb{R},$ we have 
 \begin{eqnarray*}
  \|\underline{\Psi}^r(t,x,.)\|_{H^n} &\leq& \|\underline{\Psi}^r(t,0,.)\|_{H^n} + 
  \|\underline{\Psi}^r(t,x,.)-\underline{\Psi}^r(t,0,.) \|_{H^n}\\
  &\leq& (1+\theta)\|\underline{\Psi}^r(t,0,.)\|_{H^n} + \theta\|\underline{\Psi}^r(t,x,.)\|_{H^n}+  \|R_x(t,.)\|_{H^n}+\|R_0(t,.) \|_{H^n} ,\\  
  &\leq& (1+\theta)\|\underline{\Psi}^r(t,0,.)\|_{H^n} + \theta\|\underline{\Psi}^r(t,x,.)\|_{H^n}+  \theta \|\underline{\Upsilon}(t,x,.)\|_{H^n}+\|R_0(t,.) \|_{H^n}  \\
 && +\frac{1}{2\sqrt{2}}\|F_x(t,.)\|_{H^n} .\\ 
 \end{eqnarray*}
 The fact that $\underline{\Upsilon} \in \mathcal{C}^b(\mathbb{R}^2, (H^n(\mathbb{R}^+))^2)$  (Lemma \ref{Upsilon_born}) 
 and that $\|F_x(t,.)\|_{H^n}$ does not depend on $x$
 imply that $\underline{\Psi}^r \in \mathcal{C}^b(\mathbb{R}, (H^n(\mathbb{R}^+))^2).$ To prove the limit, we use  
(\ref{00wp}) so that 
\begin{equation}\label{lim_psi}
2\sqrt{2}\underline{\Psi}^r =-\underline{\Omega}_x \underline{\Psi}^r  +\mathcal{T}_x\underline{\Psi}^r + \mathcal{T}_x\underline{\Upsilon} + F_x.
\end{equation}
We need now to calculate the limit, for fixed $(t,x)\in \R^2$, of right hand side member of  de (\ref{lim_psi}) when 
$x$ tends to $+\infty.$ We start by the term 
$$\underline{\Omega}_x \underline{\Psi}^r(x,p) = \int_0^{+\infty}\Omega(2x+s+p)\underline{\Psi}^r(x,s)ds. $$
Since $\underline{\Psi}^r \in \mathcal{C}^b(\mathbb{R}^2, (L^2(\mathbb{R}^+))^2),$ we have 
$$|\underline{\Omega}_x \underline{\Psi}^r(x,p)|^2 \leq C \int_0^{+\infty}|\Omega(2x+s+p)|^2ds\leq C(t)\sum_{k=1}^{N} \int_0^{+\infty}e^{-2\nu_k(2x+s+p)}ds,$$
and $\lim_{x\rightarrow +\infty}\bar{\Omega}_x \underline{\Psi}^r(x,p) = 0.$ Next, we have  
$$\mathcal{T}_x \underline{\Psi}^r(x,p) = \int_0^{+\infty}\hat{T}(2x+s+p)\underline{\Psi}^r(x,s)ds,$$
and 
$$|\mathcal{T}_x \underline{\Psi}^r(x,p)| \leq \|\hat{T}(2x+\cdot+p)\|_{L^2(\mathbb{R}^+)}\|\underline{\Psi}^r(x,.)\|_{L^2(\mathbb{R^+})}\leq C(t)\|\hat{T}(2x+p+\cdot)\|_{L^2(\mathbb{R}^+)}.$$
The dominated convergence theorem implies that $\lim_{x\rightarrow +\infty}\|\hat{T}(2x+p+\cdot)\|_{L^2} = 0.$ 
Similarly, we find that  
$\lim_{x\rightarrow +\infty}\mathcal{T}_x \underline{\Upsilon}(x,p) = 0.$ Finally, we have  
$$\lim_{x\rightarrow +\infty} F_x(x,p)= \lim_{x\rightarrow +\infty} \left(\begin{matrix}\widehat{d_0}(2x+p)\\ \widehat{e_0}(2x+p)\end{matrix}\right) = 
\left(\begin{matrix}0\\ 0\end{matrix}\right).$$ 
This proves that $\lim_{x\rightarrow +\infty}\underline{\Psi}^r(t,x,p)= 0.$ The remaining limit is less obvious.  
To prove it, we will prove that for every sequence $(x_k)_{k\in \mathbb{N}}$ of elements in  
$\mathbb{R}^-$ which tends to $-\infty$ when $k\rightarrow -\infty,$ there exists a subsequence 
$(x_{\gamma(k)})_{k\in \mathbb{N}}$ such that the sequence $(\underline{\Psi}(t,x_{\gamma(k)},.))_{k\in \mathbb{N}}$ converges to zero in 
$L^2_{loc}(\mathbb{R}^+).$ \\
Let $(x_k)_{k\in \mathbb{N}}$ be a sequence of negative numbers which tends to  $-\infty$ when $k\rightarrow -\infty.$ Let 
$(\Psi_k^1)_{k\in \mathbb{N}}$ be the function sequence defined by 
$$\Psi_k^1:  [2x_k, +\infty[\rightarrow \mathbb{C}^2, \Psi_k^1(p)= \underline{\Psi}^r(x_k,p-2x_k).$$
Since, for fixed $t$,  $\underline{\Psi}^r(t,.,.) \in \mathcal{C}^b(\mathbb{R}, (H^n(\mathbb{R}^+))^2)$, 
we find that there exist two functions $\Psi_\ast^1 \in (L^2(\mathbb{R}))^2$,
 $\Psi_\ast \in (H^n(\mathbb{R}^+))^2$ both depending on $t$ and a strictly increasing function   
$\gamma: \mathbb{N}\rightarrow \mathbb{N}$ such that 
$$\underline{\Psi}^r(x_{\gamma(k)},.) \rightarrow \Psi_\ast \quad \text{in}\quad (L^2_{loc}(\mathbb{R}^+))^2\quad \text{and}\quad 
\underline{\Psi}^r(x_{\gamma(k)},.) \rightharpoonup \Psi_\ast \quad \text{in}\quad (H^n(\mathbb{R}^+))^2,$$
$$\Psi^1_{\gamma(k)} \rightarrow \Psi^1_\ast \quad \text{in}\quad (L^2_{loc}(\mathbb{R}))^2\quad \text{and}\quad 
\Psi^1_{\gamma(k)} \rightharpoonup \Psi^1_\ast \quad \text{in}\quad (L^2(\mathbb{R}))^2.$$
The two functions $\underline{\Psi}(x_{\gamma(k)},.), \Psi^1_{\gamma(k)}$ satisfy
\begin{eqnarray}\label{0000wp}
 2\sqrt{2}\underline{\Psi}^r(x_{\gamma(k)},p) &=& \int_{2x_{\gamma(k)}}^{+\infty}\hat{T}(p+s)\Psi^1_{\gamma(k)}(s)ds\notag\\ 
&& +(\mathcal{T}_{x_{\gamma(k)}}\underline{\Upsilon} + F_{x_{\gamma(k)}}-\underline{\Omega}_{x_{\gamma(k)}} \underline{\Psi}^r )(x_{\gamma(k)},p),
\end{eqnarray}
\begin{eqnarray}\label{0000wp1}
 2\sqrt{2}\Psi^1_{\gamma(k)}(p) &=& \int_0^{+\infty}\hat{T}(p+s)\underline{\Psi}(x_{\gamma(k)},s)ds\notag\\ 
&& +(\mathcal{T}_{x_{\gamma(k)}}\underline{\Upsilon} + F_{x_{\gamma(k)}}-\underline{\Omega}_{x_{\gamma(k)}} \underline{\Psi}^r )(x_{\gamma(k)},p-2jx_{\gamma(k)}),
\end{eqnarray}
where
$$\underline{\Omega}_{x_{\gamma(k)}}\underline{\Psi}^r(x_{\gamma(k)},p) = 
\sum_{j=1}^Ne^{-2\nu_jx_{\gamma(k)}}A_j\left(\int_0^{+\infty}e^{-\nu_js}\underline{\Psi}^r(x_{\gamma(k)},s)ds\right)e^{-\nu_jp},$$
with $A_j = -\mu_j\left(\begin{matrix}1 & \sqrt{2}(\lambda_j+i\nu_j)\\\sqrt{2}(\lambda_j-i\nu_j) & 1 \end{matrix}\right).$
The term $\underline{\Omega}_{x_{\gamma(k)}}\underline{\Psi}^r$ converges in $(L^2_{loc}(\mathbb{R}^+))^2$
and each of its components belongs to the closed vectorial space  $ vect(e^{-\nu_1p},...,e^{-\nu_Np}).$ Then 
then there exists a family $\{c_i\}_{i=1,..,N}\in (\mathbb{C}^2)^N$ such that  
$$\bar{\Omega}_{x_{\gamma(k)}}\underline{\Psi}^r \rightarrow \sum_{j=1}^Nc_je^{-\nu_jp}\quad \text{in} \quad (L^2_{loc}(\mathbb{R}^+))^2.$$
In view of equivalence of norms 
\footnote{It suffices to take into account the two norms  $\|.\|_{L^2(I)}$ and $N_c$ defined as follows 
$$ N_c(\Phi) = \max_j(|c_j|),\quad \Phi(p) = \sum_{j=1}^Nc_je^{-\nu_jp} $$ }
on $vect(e^{-\nu_1p},...,e^{-\nu_Np}),$  we have  
$$ \lim_{k\rightarrow +\infty}\left(e^{-2\nu_jx_{\gamma(k)}}A_j\left(\int_0^{+\infty}e^{-\nu_js}\underline{\Psi}^r(x_{\gamma(k)},s)ds\right)\right) = c_j,\quad j=1,...,N,$$  
where we obtain 
$$\int_0^{+\infty}e^{-\nu_js}\Psi_\ast(s)ds= \lim_{k\rightarrow +\infty}\left(\int_0^{+\infty}e^{-\nu_js}\underline{\Psi}^r(x_{\gamma(k)},s)ds\right) = 0,\quad j=1,...,N,$$ 
Taking the limit in (\ref{0000wp}) and (\ref{0000wp1}) prove that  
\begin{equation}\label{00001wp}
2\sqrt{2}\Psi_\ast(p) = \int_{-\infty}^{+\infty}\hat{T}(p+s)\Psi_\ast^1(s)ds + \sum_{j=1}^Nc_je^{-\nu_jq}.
\end{equation}
\begin{equation}\label{00001wp1}
2\sqrt{2}\Psi_\ast^1(p) = \int_0^{+\infty}\hat{T}(p+s)\Psi_\ast(s)ds.
\end{equation}
Since $\Psi_\ast \in (vect(e^{-\nu_1p},...,e^{-\nu_Np}))^\bot$ and $\hat{T}$ is Hermitian,
\footnote{Since $c$ is assumed to be real function, we have $\bar{\beta}(t,\lambda(\xi)) =\beta(t,\lambda(-\xi)),$ 
then  $F_{22}$ $F_{12}$ are real-valued functions and  $\widehat{T}$ is  Hermitian.}  
we obtain, taking  the $L^2(\mathbb{R}^+)-$scalar product 
with $\Psi_\ast$ in (\ref{00001wp}) 
\begin{eqnarray*}
2\sqrt{2}\|\Psi_\ast\|^2_{L^2(\mathbb{R}^+)}&=& \int_0^{+\infty}\int_{-\infty}^{+\infty}\hat{T}(p+s)\Psi_\ast^1(s)\cdot\overline{\Psi_\ast}(p)dsdp\\
&=& \int_{-\infty}^{+\infty}\int_0^{+\infty}\overline{\hat{T}}(p+s)\overline{\Psi_\ast}(p)\cdot\Psi_\ast^1(s)dsdp\\
&=& 2\sqrt{2}\|\Psi_\ast^1\|^2_{L^2(\mathbb{R})}.
\end{eqnarray*}
On the other hand, we have  
$$\|\Psi_\ast\|_{L^2(\mathbb{R}^+)}\leq \frac{1}{2\sqrt{2}}\|\Psi_\ast^1\|_{L^2(\mathbb{R})}.$$ Then we finally get 
$$\Psi_\ast \equiv 0.$$
This completes the proof. 
\end{proof}

%%%%%%%%%%%%%%%%%%%%%%%%%%%%%%%%%%%%%%%%%%%%%%%%%L'etude assymptotique%%%%%%%%%%%%%%%%%%%%%%%%%%%%%%%%%%%%%

\section{Asymptotic of $N$-soliton solutions in long time: proof of Theorem \ref{GP_copm}} \label{sect:scatasympt}

We present in this section the proof of Theorem \ref{GP_copm}, which, in view of Proposition \ref{wp_Marchenko_2}, 
focuses on the analysis of solutions constructed in  Section \ref{sec_N_solit}.
Keeping the notations introduced in Section \ref{sec_N_solit},  we already proved that $(I+T)F=C$ and that 
$$QTQ^{-1}=
\frac{1}{2\sqrt{2}}\left( \begin{matrix}
E& B\\
0    & 0\\
\end{matrix} \right ),
$$
with $E=A-DB^\ast$, so that the two matrices $I+T$ and
$M=
\frac{1}{2\sqrt{2}}\left( \begin{matrix}
E+2\sqrt{2}I_N& B\\
0    & 2\sqrt{2}I_N \\
\end{matrix} \right )
$
are similar and we have $(I+T)^{-1}=Q^{-1}M^{-1}Q$.  It is easily verified that
$$
M^{-1}=2\sqrt{2}\left( \begin{matrix}
(E+2\sqrt{2}I_N)^{-1}&-\frac{1}{2\sqrt{2}}(E+2\sqrt{2}I_N)^{-1}B\\
0                    & \frac{1}{2\sqrt{2}}I_N  \\
\end{matrix} \right ).
$$
Hence, we have 
$$
(I+T)^{-1}=2\sqrt{2}\left( \begin{matrix}
(E+2\sqrt{2}I_N)^{-1}(I_N-\frac{1}{2\sqrt{2}}BD)                           &-\frac{1}{2\sqrt{2}}(E+2\sqrt{2}I_N)^{-1}B\\
D[-(E+2\sqrt{2}I_N)^{-1}(I_N-\frac{1}{2\sqrt{2}}BD)+\frac{1}{2\sqrt{2}}I_N] & \frac{1}{2\sqrt{2}}D(E+2\sqrt{2}I_E)^{-1}B+\frac{1}{2\sqrt{2}}I_N  \\
\end{matrix} \right ).
$$
Consequently, we have 
$$
F_2=\left(\begin{matrix}g_1\\\vdots\\ g_N\end{matrix}\right)=D(E+2\sqrt{2}I_N)^{-1}(-C_1+\frac{1}{2\sqrt{2}}B(DC_1+C_2))+\frac{1}{2\sqrt{2}}(DC_1+C_2),
$$
with $C_1=-\left(\begin{matrix}\mu_1e^{-\nu_1x}\\\vdots\\\mu_Ne^{-\nu_Nx}\end{matrix}\right)$ and
$C_2=-\left(\begin{matrix}\beta_1e^{-\nu_1x}\\\vdots\\ \beta_Ne^{-\nu_Nx}\end{matrix}\right).$
Thus, $DC_1+C_2=0$ and we finally get 
\begin{equation}\label{sysN}
F_2=-D(E+2\sqrt{2}I_N)^{-1}C_1.
\end{equation}

Recall that   $$\mu_k(t)=\mu_k(0)\exp(4\lambda_k\nu_kt),\quad \forall k \in \{1,2,...,N\}.$$
Before discussing the general case of $N$-solitons, we discuss the two simple cases where $N=1$ and $N=2.$
For  $N=1$, a simple calculation proves that 
$$\Psi_{12}(t,x,y)=\frac{\sqrt{2}\nu(\lambda-i\nu)e^{\nu(x-y)}}{1-2\sqrt{2}\nu\mu^{-1}(t)e^{2\nu x}}.$$
Then the $1-$soliton solution corresponding to data $\{\lambda,\mu(0)\}$ is   
\begin{equation}\label{sol1stn}
u_1(t,x)=1+2\sqrt{2}i\overline{\Psi}(t,x,x)=1+\frac{4\nu(i\lambda-\nu)}{1-2\sqrt{2}\nu\mu^{-1}(0)e^{2\nu(x-2\lambda t)}}\triangleq U(x-2\lambda t),
\end{equation}
where  $U(s)=1+\frac{4\nu(i\lambda-\nu)}{1-2\sqrt{2}\nu\mu^{-1}(0)e^{2\nu s}}$ and $\nu=\sqrt{\frac{1}{2}-\lambda^2}$.
The choice $\mu(0)=-1$ and $\lambda=\frac{1}{2}$ produces the solution
\begin{eqnarray}
u_1(t,x)=U(x-t)=\frac{i+\sqrt{2}e^{(x-t)}}{1+\sqrt{2}e^{(x-t)}}.
\end{eqnarray}
The $2-$soliton solution, $u_2$, corresponding to choices $\lambda_1=-\frac{1}{2},\lambda_2=\frac{1}{2}$ and $\mu_1(0)=\mu_2(0)=-1$ is 
$$u_2(t,x)=\frac{5\cosh(x)+3\sinh(x)+4\sqrt{2}i\sinh(t)}{5\cosh(x)+3\sinh(x)+4\sqrt{2}\cosh(t)}.$$
Taking into account the frame (in translation) with speed $2\lambda_2$ defined by the variable $\eta=x-2\lambda_2t=x-t$, we have 
$$\lim_{t\rightarrow +\infty}u_2(t,\eta+t)=\frac{i+\sqrt{2}e^\eta}{1+\sqrt{2}e^\eta}=U(\eta),$$
$$\lim_{t\rightarrow -\infty}u_2(t,\eta+t)=-i\frac{i+2\sqrt{2}e^\eta}{1+2\sqrt{2}e^\eta}=-iU(\eta+\log(2)),$$
i.e. the solution $u_2$ behaves (asymptotically when $t \rightarrow \pm \infty$) 
as a progressive wave of speed  $2\lambda_2$.

Now let us come back to the general case ($N\in \N$), for which we can not establish an explicit formula for the solution $u(t,x)$. 
To this end, we rewrite
\begin{eqnarray*}
u(t,x)&=&1+2\sqrt{2}i\overline{\Psi}_{12}(t,x,x)\\
&=&1+2\sqrt{2}i\sum_{j=1}^{N}\bar{g}_j(t,x)\exp(-\nu_jx)\\
&=&1+2\sqrt{2}i\sum_{j=1}^{N}\overline{G}_j(t,x).
\end{eqnarray*}
Hence, we need to study the behavior of the sum $\sum_{j=1}^{N}G_j$. To this end rewrite the system (\ref{sysN}) as follows
\begin{equation}\label{sysNbis}
(E+2\sqrt{2}I_N)D^{-1}F_2=-C_1,
\end{equation}
with
$$(E)_{kj}=-(\mu_k+\bar{\beta}_k\frac{\beta_j}{\mu_j})\frac{\exp(-(\nu_k+\nu_j)x)}{\nu_k+\nu_j},$$
$$(D)_{kj}=-\delta_{kj}\frac{\beta_j}{\mu_j}\quad 1\leq k,j\leq N,$$
where  $$\delta_{kj}=\left\{\begin{array}{lr}
              1,&\text{if}\quad k=j\\
              0,& \text{if} \quad k\neq j.
             \end{array}
\right.$$
The system (\ref{sysNbis}) can be rewritten as follows
$$\sum_{j=1}^{N}\left(-2\sqrt{2}\delta_{kj}\frac{\mu_j}{\beta_j}g_j+(\bar{\beta}_k+\mu_k\frac{\mu_j}{\beta_j})\frac{e^{-(\nu_k+\nu_j)x}}{\nu_k+\nu_j}g_j\right)
=\mu_ke^{-\nu_kx},\quad k=1,2,...,N.$$
which implies 
$$-\frac{2\sqrt{2}}{\beta_k}\exp(2\nu_kx)G_k(t,x)+\sum_{j=1}^{N}\left(\frac{\bar{\beta}_k}{\mu_k}+\frac{\mu_j}{\beta_j}\right)\frac{G_j(t,x)}{\nu_k+\nu_j}=1,$$
with $\beta_k=\sqrt{2}\mu_k(\lambda_k-i\nu_k)$, $\mu_k(t)=\mu_k(0)\exp(4\lambda_k\mu_kt)$
and $\frac{\bar{\beta}_k}{\mu_k}+\frac{\mu_j}{\beta_j}=\sqrt{2}(\lambda_k+\lambda_j+i(\nu_k+\nu_j)).$
Finally, we obtain
\begin{equation}\label{sysNter}
-\frac{2\mu_k(0)^{-1}}{\lambda_k-i\nu_k}\exp(2\nu_k(x-2\lambda_kt))G_k(t,x)+\sqrt{2}\sum_{j=1}^{N}\left(\frac{\lambda_k+\lambda_j}{\nu_k+\nu_j}+i\right)G_j(t,x)=1\quad k=1,2,...,N.
\end{equation}
Let $k_0\in \{1,...,N\}$. We define the frame (in translation)  with speed $2\lambda_{k_0}$ by the variable
$$\eta=x-2\lambda_{k_0}t.$$
We will analyze the asymptotic in long time. Note that 
\begin{eqnarray*}
-\frac{2\mu_k(0)^{-1}}{\lambda_k-i\nu_k}\exp(2\nu_k(x-2\lambda_kt))&=&-\frac{2\mu_k(0)^{-1}}{\lambda_k-i\nu_k}\exp(2\nu_k(x-2\lambda_{k_0}t+2\lambda_{k_0}t-2\lambda_kt))\\
&=&-\frac{2\mu_k(0)^{-1}}{\lambda_k-i\nu_k}\exp(2\nu_k\eta)\exp(4\nu_k(\lambda_{k_0}-\lambda_k)t)\\
&=&H_k(\eta)\exp(4\nu_k(\lambda_{k_0}-\lambda_k)t).
\end{eqnarray*}
The system (\ref{sysNter})  can be  transformed into 
$$H_k(\eta)\exp(4\nu_k(\lambda_{k_0}-\lambda_k)t)G_k(t,x)+\sqrt{2}\sum_{j=1}^{N}(\frac{\lambda_k+\lambda_j}{\nu_k+\nu_j}+i)G_j(t,x)=1,\quad k=1,2,...,N.$$
We take the limit when $t\rightarrow -\infty$ for fixed $\eta$. This leads, denoting 
$G_k=\lim_{t\rightarrow -\infty}G_k(t,\eta+2\lambda_{k_0}t)$ to
\begin{equation}\label{sysk0}
  \left\{
  \begin{array}{lr}
  \sum_{j=1}^{k_0}(\frac{\lambda_k+\lambda_j}{\nu_k+\nu_j}+i)G_j=1-\delta_{kk_0}H_{k_0}(\eta)G_{k_0},& k=1,...,k_0 \\
 G_k=0,& k=k_0+1,...,N.
  \end{array}
  \right.
\end{equation}
We introduce the matrices 
$$K=\sqrt{2}\left(\frac{\lambda_k+\lambda_j}{\nu_k+\nu_j}+i\right)_{1\leq k,j\leq k_0},\quad K^-=\sqrt{2}\left(\frac{\lambda_k+\lambda_j}{\nu_k+\nu_j}+i\right)_{1\leq k,j\leq k_0-1},$$
and $\widetilde{K}$ obtained from $K$ by replacing the last column by a column of ones.
Using the Cramer formula  for (\ref{sysk0}),  we have 
\begin{eqnarray*}
G_k&=&\frac{\sum_{l=1}^{k_0}(1-\delta_{lk_0}H_{k_0}G_{k_0})K_{lk}}{\det(K)}\\
&=&\frac{\sum_{l=1}^{k_0}K_{lk}-H_{k_0}G_{k_0}K_{k_0k}}{\det(K)},\quad k=1,...,k_0,
\end{eqnarray*}
where the  $K_{lk}$ refer to the cofactors of $K$. For $k=k_0$, we also have 
$$G_{k_0}=\frac{\det(\widetilde{K})}{\det(K)+\det(K^-)H_{k_0}},$$
so that 
$$G_k=\frac{\sum_{l=1}^{k_0}K_{lk}}{\det(K)}-\frac{\det(\widetilde{K})H_{k_0}}{\det(K)^2+\det(K^-)\det(K)H_{k_0}}K_{k_0k}.$$
Summing the last  equation on $k$, we obtain 
$$\sum_{k=1}^{k_0}G_k=\frac{\sum_{k=1}^{k_0}\sum_{l=1}^{k_0}K_{lk}}{\det(K)}-\frac{\det(\widetilde{K})H_{k_0}}{\det(K)^2+\det(K^-)\det(K)H_{k_0}}\sum_{k=1}^{k_0}K_{k_0k}.$$
Thus, we finally get 
$$\sum_{k=1}^{k_0}G_k=\frac{\sum_{1\leq k,l\leq k_0}K_{lk}}{\det(K)}-\frac{\det(\widetilde{K})^2H_{k_0}}{\det(K)^2+\det(K^-)\det(K)H_{k_0}},$$
and
\begin{equation}
\lim_{t\rightarrow-\infty}u(t,\eta+2\lambda_{k_0}t)=
\frac{\beta(\beta+\alpha)-\frac{2\mu_{k_0}(0)^{-1}}{\lambda_{k_0}+i\nu_{k_0}}(\delta(\beta+\alpha)-\gamma)e^{2\nu_{k_0}\eta}}
{\beta^2-\frac{2\mu_{k_0}(0)^{-1}}{\lambda_{k_0}+i\nu_{k_0}}\delta\beta e^{2\nu_{k_0}\eta}},
\end{equation}
with
\begin{equation*}\left\{\begin{array}{lr}
\alpha=2\sqrt{2}i\sum_{1\leq k,l\leq k_0}\overline{K}_{lk},&\\
\beta=\overline{\det(K)},&\\
\gamma=2\sqrt{2}i\overline{\det^2(\widetilde{K})},&\\
\delta=\overline{\det(K^-)}.&
\end{array}\right.
\end{equation*}
\begin{proposition}\label{Nstn-1stn}
There exist $A^-\in\mathbb{C}$ and $\eta_{k_0}^- \in \mathbb{R}$ such that 
$$|A^-|=1,$$
$$\lim_{t\rightarrow -\infty}u(t,\eta+2\lambda_{k_0}t)=A^-U(\eta-\eta_{k_0}^-),$$
where $U$ is the function defined in (\ref{sol1stn}) which corresponds to $\mu_{k_0}(0)$ and $\lambda_{k_0}$.
\end{proposition}
\begin{proof}
Recall first that $$U(\eta)=\frac{4\nu_{k_0}(i\lambda_{k_0}-\nu_{k_0})+1-2\sqrt{2}\nu_{k_0}\mu_{k_0}^{-1}(0)e^{2\nu_{k_0}\eta}}{1-2\sqrt{2}\nu_{k_0}\mu_{k_0}^{-1}(0)e^{2\nu_{k_0}\eta}},$$
and that 
\begin{equation}
\lim_{t\rightarrow-\infty}u(t,\eta+2\lambda_{k_0}t)=
A^-\frac{4\nu_{k_0}(i\lambda_{k_0}-\nu_{k_0})+1-\frac{2\mu_{k_0}(0)^{-1}(4\nu_{k_0}(i\lambda_{k_0}-\nu_{k_0})+1)}{\lambda_{k_0}+i\nu_{k_0}}(\delta/\beta-\gamma/\beta(\alpha+\beta))e^{2\nu_{k_0}\eta}}
{1-\frac{2\mu_{k_0}(0)^{-1}}{\lambda_{k_0}+i\nu_{k_0}}\frac{\delta}{\beta} e^{2\nu_{k_0}\eta}},
\end{equation}
with $A^-=\frac{\alpha+\beta}{\beta(4\nu_{k_0}(i\lambda_{k_0}-\nu_{k_0})+1)}$.
Then It will be required to prove:\\
(i)\quad$|A^-|=1$.\\
(ii)\quad$\delta=[\delta-\frac{\gamma}{\alpha+\beta}](4\nu_{k_0}(i\lambda_{k_0}-\nu_{k_0})+1)$.\\
(iii)\quad$\frac{\delta}{\beta(\lambda_{k_0}+i\nu_{k_0})}\in \mathbb{R}^+$.\\
In what concerns (i), remark that $|4\nu_{k_0}(i\lambda_{k_0}-\nu_{k_0})+1|=1$ 
(since $\nu_{k_0}=\sqrt{\frac{1}{2}-\lambda_{k_0}}$ ), and the fact that $|\frac{\alpha+\beta}{\beta}|=1$ follows from 
the following lemma, since
\begin{eqnarray*}
\overline{\alpha+\beta} &=& \det(K)-2\sqrt{2}i\sum_{1\leq i,j\leq k_0}K_{ij}\\
&=& \det(K-\sqrt{2}i J_{k_0})+\sqrt{2}i\sum_{1\leq i,j\leq k_0}K_{ij}-2\sqrt{2}i\sum_{1\leq i,j\leq k_0}K_{ij}\\
&=& \det(K-\sqrt{2}i J_{k_0})-\sqrt{2}i\sum_{1\leq i,j\leq k_0}K_{ij}\\
&=& \beta,
\end{eqnarray*}
%%%%%%%%%%%%%%%%%%%%%%%%%%%%%%%%%%%%%%%%%%%%%%%%%%%%%%%%%%%%%%%%%%%%%%%%%%%%%%%%%%%%%%%%%%%%%%%%%%%%%%%%%%%%%%%%%%%%%%%%%%%%%%%%%%%%%%%
\begin{lemma}\label{lem-mat-cof}
Let $M\in \mathcal{M}_N(\mathbb{C})$ and $J_N$ the matrix whose coefficients are ones. Denote $K=M+XJ$. Then  
$$\det(K)=\det(M)+X\sum_{1\leq i,j\leq N}K_{ij},$$
where the $K_{ij}$ refer to the cofactors of $K$.
\end{lemma}
\begin{proof}
Let $U,V\in \mathbb{C}^N$ be two column vectors. Then  
$$\det(I_N+U.V^t)=1+\langle U,V\rangle,$$
where $\langle,\rangle$ is the usual scalar product in $\mathbb{C}^N.$ In fact, it suffices to expand the operator  
$T=I+U.V^t$ in the basis composed of $V$ and the basis of the orthogonal to $V$.
Assume now that the matrix $K=M+U.V^t$ is invertible. Then we can write 
\begin{eqnarray*}
\det(M)&=& \det(K-U.V^t)\\
&=& \det(K)\det(I_N-K^{-1}U.V^t)\\
&=& \det(K)(1-\langle K^{-1}U,V\rangle)\\
&=& \det(K)-\langle\left(\det(K)K^{-1}U\right),V\rangle\\
&=& \det(K)-\langle Cof(K)U,V\rangle,
\end{eqnarray*}
where $Cof(K)$ is the cofactor matrix of $K$.
The conclusion therefore follows by taking $U=(1,1,...,1)^t$ and $V=(X,X,...,X)^t$. 
Finally, the case $K$ non-invertible is obtained by continuity. 
\end{proof}

For equation (ii), it suffices to prove that  
$$\sqrt{2}(\frac{\lambda_{k_0}}{\nu_{k_0}}-i)\det(\widetilde{K})^2=\det(K^-)\overline{\det(K)}.$$
Since we have assumed that $\quad -\frac{1}{\sqrt{2}}\leq\lambda_k\leq\frac{1}{\sqrt{2}},\quad \forall k\in\{1,2,...,N\},$ 
there exists $\theta_k\in[0,\pi]$ such that
$$\lambda_k=\frac{1}{\sqrt{2}}\cos(\theta_k) \qquad\text{et}\qquad \nu_k=\frac{1}{\sqrt{2}}\sin(\theta_k).$$ We also have
$$\left(K\right)_{kj}=\sqrt{2}\left(\frac{\lambda_k+\lambda_j}{\nu_k+\nu_j}+i\right)=\frac{e^{i\theta_k}+e^{\theta_j}}{\sin(\theta_k)+\sin(\theta_j)}=\frac{e^{i(\frac{\theta_k+\theta_j}{2})}}{\sin(\frac{\theta_k+\theta_j}{2})}.$$
Subtracting column $k_0$ from column $j$ in the matrix $\overline{K}$, we get
\begin{eqnarray*}
\forall k \in \{1,...,k_0\}\quad (\overline{K})_{kj}-(\overline{K})_{kk_0}
&=&\frac{\sin(\frac{\theta_k+\theta_{k_0}}{2})e^{-i(\frac{\theta_j-\theta_{k_0}}{2})}-\sin(\frac{\theta_k+\theta_j}{2})}{\sin(\frac{\theta_k+\theta_j}{2})}(\overline{K})_{kk_0}\\
&=&\frac{\sin(\frac{\theta_k+\theta_{k_0}}{2})-\sin(\frac{\theta_k+\theta_j}{2})e^{i(\frac{\theta_j-\theta_{k_0}}{2})}}{\sin(\frac{\theta_k+\theta_j}{2})}e^{-i(\frac{\theta_j-\theta_{k_0}}{2})}(\overline{K})_{kk_0}\\
&=&\frac{e^{i(\frac{\theta_j+\theta_k}{2})}\sin(\frac{\theta_{k_0}-\theta_j}{2})}{\sin(\frac{\theta_k+\theta_j}{2})}e^{-i(\frac{\theta_j-\theta_{k_0}}{2})}(\overline{K})_{kk_0}\\
&=&\left(e^{-i\frac{\theta_j}{2}}\sin(\frac{\theta_{k_0}-\theta_j}{2})\right)
\left(\frac{e^{-i\frac{\theta_k}{2}}}{\sin(\frac{\theta_{k_0}+\theta_k}{2})}\right)(\widetilde{K})_{kj},\\
\end{eqnarray*}
which implies\footnote{We have used the identity $\sin(a+c)=\sin(a+b)e^{i(b-c)}+\sin(c-b)e^{i(a+b)}.$} that  
\begin{equation}\label{det1}
\det(\overline{K})=\frac{\sqrt{2}}{\sin(\theta_{k_0})}\prod_{j=1}^{k_0-1}\frac{\sin(\frac{\theta_{k_0}-\theta_j}{2})}{\sin(\frac{\theta_{k_0}+\theta_j}{2})}
e^{-i\sum_{j=1}^{k_0}\theta_j}\det(\widetilde{K}).
\end{equation}
A similar argument proves that
\begin{equation}\label{det2}
\det(\widetilde{K})=\prod_{j=1}^{k_0-1}\frac{\sin(\frac{\theta_{k_0}-\theta_j}{2})}{\sin(\frac{\theta_{k_0}+\theta_j}{2})}e^{-i\sum_{j=1}^{k_0-1}\theta_j}\det(K^-).
\end{equation}
Combining (\ref{det1}) and (\ref{det2}) we complete the proof of (ii).

We focus now on assertion (iii). Subtracting the line $k_0$ from each of the previous lines, we find that $\det(\widetilde{K})\in\mathbb{R}$
and using (ii), we obtain
\begin{eqnarray*}
\frac{\bar{\delta}}{\bar{\beta}(\lambda_{k_0}-i\nu_{k_0})}&=& \frac{\det(K^-)}{\det(K)(\lambda_{k_0}-i\nu_{k_0})}\\
&=&\frac{\sqrt{2}\det(\widetilde{K})^2}{\nu_{k_0}|\det(K)|^2}\in \mathbb{R}^+_\ast.\\
\end{eqnarray*}
This allows to define  $\eta_{k_0}^-$ by the  following relationship
\begin{eqnarray*}
e^{-2\nu_{k_0}\eta_{k_0}^-}&=&\frac{\det(\widetilde{K})^2}{\nu_{k_0}^2|\det(K)|^2}\\
&=&\prod_{j=1}^{k_0-1}\frac{\sin^2(\frac{\theta_{k_0}+\theta_j}{2})}{\sin^2(\frac{\theta_{k_0}-\theta_j}{2})}\\
&=&\prod_{j=1}^{k_0-1}\frac{1-\cos(\theta_{k_0}+\theta_j)}{1-\cos(\theta_{k_0}-\theta_j)}\\
&=&\prod_{j=1}^{k_0-1}\frac{1-\cos(\theta_{k_0})\cos(\theta_j)+\sin(\theta_{k_0})\sin(\theta_j)}
{1-\cos(\theta_{k_0})\cos(\theta_j)-\sin(\theta_{k_0})\sin(\theta_j)}\\
&=&\prod_{j=1}^{k_0-1}\frac{1-2\lambda_{k_0}\lambda_j+2\nu_{k_0}\nu_j}
{1-2\lambda_{k_0}\lambda_j-2\nu_{k_0}\nu_j}.
\end{eqnarray*}
Finally, we get
\begin{eqnarray*}
A^-&=&\frac{e^{2i\sum_{j=1}^{k_0}\theta_j}}{4\nu_{k_0}(i\lambda_{k_0}-\nu_{k_0})+1}\\
&=& \frac{e^{2i\sum_{j=1}^{k_0}\theta_j}}{\cos(2\theta_{k_0})+i\sin(2\theta_{k_0})}\\
&=&e^{2i\sum_{j=1}^{k_0-1}\theta_j}.
\end{eqnarray*}
\end{proof}

A parallel argument allows, for the limit $t\rightarrow +\infty$, to obtain
$$\lim_{t\rightarrow +\infty}u(t,\eta+2\lambda_{k_0}t)=A^+U(\eta-\eta_{k_0}^+),$$
where this time $A^+$ and $\eta_{k_0}^+$ are defined as follows
$$e^{-2\nu_{k_0}\eta_{k_0}^+}=\prod_{j=k_0+1}^{N}\frac{1-2\lambda_{k_0}\lambda_j+2\nu_{k_0}\nu_j}
{1-2\lambda_{k_0}\lambda_j-2\nu_{k_0}\nu_j}$$
and $$A^+=e^{2i\sum_{j=k_0+1}^{N}\theta_j}.$$
This completes the proof of Theorem \ref{GP_copm}. 

\bibliographystyle{plain}
\bibliography{biblio3.bib}
\end{document}